\documentclass[12pt]{amsart}
\usepackage{amsmath}
\usepackage{amsthm}
\usepackage{amssymb}
\usepackage{amsmath}
\usepackage{amsthm}
\usepackage{amssymb}
\usepackage{fancyhdr}
\usepackage{enumerate}
\usepackage{amscd}
\usepackage[all,cmtip]{xy}
\usepackage{graphicx}
\usepackage{color}
\usepackage{hyperref}
\usepackage{cancel}
\usepackage[normalem]{ulem}

\hypersetup{
    colorlinks,
    citecolor=black,
    filecolor=black,
    linkcolor=black,
    urlcolor=black
}
\usepackage{tikz}
\usepackage[all]{xy}
\usepackage{graphicx}

\usepackage{tikz-cd}

\usetikzlibrary{backgrounds,arrows}

\newtheorem{theorem}{Theorem}[section]
\newtheorem{corollary}[theorem]{Corollary}
\newtheorem{lemma}[theorem]{Lemma}
\newtheorem{proposition}[theorem]{Proposition}

\newtheorem{claim}[theorem]{Claim}

\theoremstyle{definition}
\newtheorem{definition}[theorem]{Definition}

\theoremstyle{remark}
\newtheorem{remark}[theorem]{Remark}

\numberwithin{equation}{section}

\newcommand{\R}{\mathbb R}
\newcommand{\N}{\mathbb N}

\newcommand{\NN}{\mathcal N}
\newcommand{\C}{\mathbb C}
\newcommand{\CC}{\mathcal C}
\newcommand{\VV}{\mathcal V}
\newcommand{\id}{\text{id}}
\newcommand{\TT}{\mathcal T}
\newcommand{\Aa}{\mathcal A}
\newcommand{\Ss}{\mathcal S}
\newcommand{\LL}{\mathcal L}
\newcommand{\JJ}{\mathcal J}
\newcommand{\ZZ}{\mathcal Z}
\newcommand{\QQ}{\mathcal Q}
\newcommand{\D}{\mathbb D}
\newcommand{\Sp}{\mathbb S}

\newcommand{\B}{\mathbb B}
\newcommand{\Z}{\mathbb Z}
\newcommand{\BB}{\mathcal B}

\newcommand{\HH}{\mathcal H}
\newcommand{\EE}{\mathcal E}
\newcommand{\DD}{\mathcal D}
\newcommand{\OO}{\mathcal O}
\newcommand{\UU}{\mathcal U}

{{\sc Proof of Theorem~\ref{cannon}.}}%
{{\qed} \\}

{{\sc Proof of Theorem~\ref{degree}.}}%
{{\qed} \\}

{{\sc Proof of Proposition~\ref{same}.}}%
{{\qed} \\}

\newenvironment{proofSphereCor}%
{{\sc Proof of Theorem \ref{sphere_cor}.}}%
{{\qed} \\}

{{\sc Proof of (C2).}}%
{{\qed} \\}

{{\sc Proof of Theorem \ref{Main2}.}}%
{{\qed} \\}

\newenvironment{proofMain1}%
{{\sc Proof of Theorem \ref{Main1}.}}%
{{\qed} \\}

\newenvironment{proofMain3}%
{{\sc Proof of Theorem~\ref{Main3}.}}%
{{\qed} \\}

{{\sc Proof of Theorem~\ref{cannon}.}}%
{{\qed} \\}

{{\sc Proof of Corollary~\ref{HomeoCor}.}}%
{{\qed} \\}

\allowdisplaybreaks

\date{\today}
\title[Uniformization of CAT($k
$) spheres]
{Harmonic branched coverings and \\ uniformization of CAT($k$) spheres}

\thanks{
CB supported in part by NSF DMS-1609198 and NSF DMS CAREER-1750254, CM supported in part by NSF DMS-2005406.}
\author[Breiner]{Christine Breiner}
\address{Department of Mathematics \\
                 Fordham University \\
                 Bronx, NY  10458}
\email{cbreiner@fordham.edu}

\author[Mese]{Chikako Mese}
\address{Johns Hopkins University\\
Department of Mathematics\\
3400 N. Charles Street\\
Baltimore, MD  21218}
\email{cmese@math.jhu.edu}

\begin{document}

\maketitle

\begin{abstract}
Let $S$ be a surface with a metric $d$ satisfying an upper curvature bound in the sense of Alexandrov (i.e.~via triangle comparison).  
We show that an almost conformal harmonic map from a surface into $(S,d)$ is a branched covering.  As a consequence, if $(S,d)$ is homeomorphically equivalent to the 2-sphere $\Sp^2$, then it is conformally equivalent to $\Sp^2$.    
MSC 58E20, 30F10
\end{abstract}

\section{Introduction} \label{intro}

The uniformization theorem for Riemann surfaces was one of the landmark achievements in the mathematics of the 19th and early 20th century.   Due to Koebe and Poincar\'e, and building on  prior works of Gauss, Abel, Jacobi, Riemann, Weierstrass, Clebsch, Fuchs, Schwarz, Klein, Fricke, Hilbert and Osgood among others, the theorem asserts that \emph{every simply connected Riemann surface is conformally equivalent to one of three Riemann surfaces: the open unit disk, the complex plane, or the Riemann sphere}.  
The result and its various proofs have had a major impact on several fields of mathematics, including complex analysis, geometry, combinatorial group theory and topology.  In geometry for instance, the uniformization theorem implies that every smooth Riemannian metric $g$ defined on a closed surface  $S$ is   conformally equivalent to a Riemannian metric $g_0$  of constant Gaussian curvature; i.e.~there exists a diffeomorphism $u:S \rightarrow S$ and a positive function $\lambda$  such that the pullback $u^*g$ of $g$ via $u$ satisfies $u^*g=\lambda g_0$.   

In the past few decades, there has been a growing  interest in non-smooth spaces and in their corresponding analysis.  From this perspective, it is natural to examine the uniformization of non-smooth geometry and, in particular,   ask when a geometric space is conformally equivalent to a Riemannian metric of constant Gaussian curvature.   An example of a result of this type can be deduced from a classical result of Ahlfors-Bers \cite{ahlfors-bers} and Morrey~\cite{morrey}.  Indeed, the \emph{Bounded Measurable Riemann Mapping Theorem}, which generalizes the uniformization theorem, implies the following:    
\emph{If $g$  is a bounded measurable Riemannian metric on the 2-sphere ${\mathbb S}^2$, then
 there  exists a quasiconformal map  $u:{\mathbb S}^2 \rightarrow ({\mathbb S}^2, g)$ from the standard 2-sphere that is conformal almost everywhere and unique up to composition with a M\"{o}bius transformation.}  Here, by the standard 2-sphere, we mean the topological sphere   $\Sp^2=\{(x,y,z) \in \R^3:   x^2+y^2+z^2=1\}$ endowed with the  metric $g_{\Sp^2}$ inherited from the embedding  $\Sp^2 \hookrightarrow \R^3$.  To the extent of our knowledge,  Y.~Reshetnyak \cite{reshetnyak0} was the first to address the question of conformal parameterization of metric spaces.  The method employed in \cite{reshetnyak0} is to take an approximation of a  singular surface by piecewise linear surfaces and the local approach there differs from the global approach taken up here using harmonic maps.  The local conformal parameterization problem for metric spaces was further studied by the second author in \cite{mese-thesis} and \cite{meseMM} via a harmonic maps approach.     
A current active area of study is the  quasiconformal equivalency of the sphere, i.e.~ the question of when a metric space which is topologically equivalent to a sphere is quasiconformally equivalent to the Riemann sphere (e.g.~\cite{bonk-kleiner}, \cite{lytchak-wenger}, \cite{rajala}).  
Although not explicitly stated in his work with S.~Wenger, A.~Lytchak \cite{lytchak-private} has explained to us how $1$-quasiconformality of the quasiconformal map can be shown for spaces ``which do not contain infinitesimal non-Euclidean norms".  We also add that the harmonic maps approach with an applied mathematics bent has been studied by several authors.  For more detail on the work in this area, we refer to the survey paper of X.~Gu, F.~Luo and S.~T.~Yau \cite{guluoyau} and the references therein. 

In this paper,  we take a different approach to the uniformization problem than the ones taken in the aforementioned work. 
Our focus is on the branched covering and uniformization of surfaces endowed with a distance function  satisfying an upper curvature bound given in terms of the   CAT($\kappa$) inequality. This means that sufficiently small geodesic triangles are ``skinnier" than a corresponding comparison triangle in a Riemannian surface of constant Gaussian curvature $\kappa$.  In particular, our construction of a conformal map relies on the following:  (i) The generalization of the Sacks-Uhlenbeck bubbling by the authors and their collaborators \cite{paper2}, namely the existence of a harmonic map from a compact surface to a CAT($\kappa$) space, and   (ii) A careful local analysis of the harmonic map when the domain and the target spaces are both (topologically) the 2-sphere.  The analysis in (ii) allows us to conclude that the harmonic map from (i) is in fact a branched cover. We construct a 1-quasiconformal map by taking a quotient of this branched cover.

In order to elaborate on the existence statement of item (i), we recall the following deep theorem of Sacks and Uhlenbeck \cite{sacks-uhlenbeck}:  \emph{Given a finite energy map from a Riemann surface into a compact Riemannian manifold,  either there exists a  harmonic map homotopic to the given map  or there exists a branched minimal immersion of the 2-sphere.}
The existence theory of harmonic maps when the target space has non-positive curvature has been widely addressed.  However, the existence without the upper curvature bound of 0 is much more complicated, and this result of Sacks-Uhlenbeck was a breakthrough in the field.  Indeed, their study of the ``bubbling phenomena,"  that either a minimizing sequence of maps converges to a harmonic map or forms a ``bubble" (i.e.~a harmonic map from a sphere) has been a widely influential idea in geometric analysis.    The authors of the current article and their collaborators generalized  the Sacks-Uhlenbeck theorem in the metric space setting and proved the following  \cite{paper2}:  \\
 \\
{\bf Theorem.}  \emph{If $\Sigma$ is a compact Riemann surface, $(X,d)$ is a compact locally CAT($\kappa$) space, and $\varphi:\Sigma \rightarrow X$ is a continuous finite energy map, then either there exists a harmonic map $u:\Sigma \rightarrow (X,d)$ homotopic to $\varphi$ or an almost conformal harmonic map $v:\mathbb S^2 \rightarrow (X,d)$.}\\

On the one hand, by applying the above theorem with $\Sigma=\Sp^2$, either of the alternatives  yields a harmonic map from the standard 2-sphere.   On the other hand, proving uniformization requires the existence of a harmonic \emph{homeomorphism} and it is unclear that the second alternative in the theorem yields even a degree 1 map (even when the initial map $\varphi$ is of degree 1).  Thus,   further  analysis  of this harmonic map  is needed. Note that the first author and S. Lakzian \cite{breinerlakzian} completed the full bubble tree picture for harmonic maps into compact locally CAT($\kappa$) spaces, but this compactness result also fails to guarantee the existence of a {homeomorphism}.

The second ingredient (i.e.~item (ii)), the analysis of the local behavior of harmonic maps through its   tangent maps, is the main technical accomplishment of this paper. The use of the tangent map as a tool in the analysis of harmonic maps in the singular setting was initiated in the seminal work of Gromov and Schoen \cite{gromov-schoen} and also developed for example in   \cite{daskal-meseCAG}, \cite{daskal-meseFPT}, \cite{daskal-meseMem}, \cite{daskal-meseCre}, \cite{daskal-meseHR}, \cite{daskal-meseGAFA}, \cite{BFHMSZ}.  We  advance this idea further in the setting of CAT($\kappa$) surfaces.  Using tangent maps, we define a notion of a non-degenerate harmonic map in this setting of singular surfaces, generalizing harmonic diffeomorphisms between smooth surfaces.  
We prove that a non-degenerate harmonic map is well-behaved locally  and thus defines a branched covering.  
\begin{theorem}\label{Main1} 
A proper, non-degenerate  harmonic map from a Riemann surface to an oriented locally CAT($\kappa$) surface is a branched cover; i.e.~the map  is a covering map away from a  discrete subset of the Riemann surface.
If the map is degree 1, then the map is a homeomorphism.
\end{theorem}
 
Specializing to the case when the domain is the standard sphere, we obtain the following: 

 \begin{theorem}\label{sphere_cor}
A  non-trivial harmonic map from the standard sphere $\Sp^2$ to an oriented locally CAT($\kappa$) sphere is  an almost conformal branched cover.  If the map is degree 1, then it is a conformal (i.e.~1-quasiconformal) homeomorphism with conformal inverse.
 \end{theorem}

As a consequence, we obtain a uniformization theorem. That is, if the CAT($\kappa$) space is  homeomorphic to a sphere then it is conformally equivalent to the standard sphere $\Sp^2$.  
 \begin{theorem}  \label{Main3}
   If $(S,d)$ is a locally CAT($\kappa$) sphere, then there exists a conformal (i.e.~1-quasiconformal) harmonic homeomorphism $h:\Sp^2 \rightarrow (S,d)$ from the standard sphere, with $h^{-1}$ also conformal, which is  unique up to composition with a M\"obius transformation.  
   Moreover, $h$ is almost conformal and the energy of the map is twice the Hausdorff 2-dimensional measure of $(S,d)$.
   \end{theorem}
   The notion of \emph{conformality} (a.k.a.~1-quasiconformality)  is in the metric space sense and captures  the property that \emph{infinitesimal circles are transformed to infinitesimal circles} (cf.~Definition~\ref{qc_def}).   
Our theorem asserts more than conformal equivalence of the two spaces.  Indeed, Theorem~\ref{Main3} asserts that the conformal equivalence is achieved by an \emph{almost conformal harmonic} map.  The notion of an almost conformal map 
 captures the geometric property  that  the pullback metric of $h$ is  conformally equivalent to a Riemannian metric of constant curvature (cf.~Definition~\ref{def:conformal}).\\
\\ 
{\sc Main ideas and Outline of the paper:}
\vskip .25cm

The paper is roughly divided into two parts:
\begin{itemize}
\item  {\sc Part I}:  Local analysis and branched covering results~(Sections 2-4)
\item {\sc Part II}:  Existence and uniqueness results~(Section 5)
\end{itemize}
\vskip .15cm

{\sc Part I}.   The proof of 
Theorem~\ref{Main1} relies on the analysis of the local behavior of harmonic maps.   
The main tool for this  is the Alexandrov  tangent maps associated to a harmonic map whose usefulness is due to the fact that they map into  tangent cones of the target CAT($\kappa$) space.  (This is in analogy to the differentials of a smooth map between  Riemannian manifolds which map into tangent spaces.)  In comparison,  a (non-Alexandrov) tangent map of  a  harmonic map into an arbitrary CAT($\kappa$) space maps into an abstract metric space that is not necessarily a tangent cone, as indicated in  \cite[Section 3]{korevaar-schoen2}.     Generally speaking,  an Alexandrov tangent map is not necessarily harmonic. Consider the following example.\\
\\
{\it Example.}  First, let $({\bf H}, g_{\bf H})$ be the smooth Riemannian surface given by 
 \[
 {\bf H}=\{(x,y) \in \R^2: y >0\}, \ \ 
g_{\bf H}=dy^2+ y^6 dx^2.
\]
(It is instructive to think of this surface as the covering space of a cusp minus the cusp point, or more precisely, as a covering space of a  surface of revolution in $\R^3$ of the planar curve $y=x^{\frac{1}{3}}$ minus the origin.)  Next, let  $\overline{\bf H}$ be the  metric completion of ${\bf H}$ constructed by adding the boundary line $\{y=0\}$ and identifying this line as a single point $P_0$.  This is a CAT(0) space (and an important object in the study of Teichm\"uller space, cf.~\cite{daskal-meseHR} and references therein).  A vertical line $\{x=c\}$ for a fixed  constant $c \in \R$ is a geodesic emanating from $P_0$ and the angle between  any two such geodesics at $P_0$ is equal to 0.    Thus,  the space of directions at $P_0$ (i.e.~the equivalence class of geodesics where two geodesics are equivalent if and only if the angle between them is 0) has only one element.  This implies that  the tangent cone $T_{P_0}\overline{\bf H}$ (i.e.~a metric cone over the space of directions) is isometric to  the interval $[0,\infty)$, and  an Alexandrov  tangent map of harmonic map $u$  into $\overline{\bf H}$ at a point in $u^{-1}(P_0)$ can be viewed as a  \emph{function} mapping into the interval $[0,\infty)$.  This Alexandrov tangent  map  cannot be a harmonic  because otherwise it would violate the minimum principle for harmonic functions by having 0 in its range.  \\

The situation for a harmonic map into a  CAT($\kappa$) surface is different than the above example since any tangent cone is a metric cone over a closed curve and does not allow for pathological tangent cones as in  the example above (cf.~Proposition~\ref{coneovercurve}). Indeed, we show that the Alexandrov tangent maps of a harmonic map into a  CAT($\kappa$) surface are harmonic (cf.~Theorem~\ref{tangent_maps_are_cones}). 
Thus, we can characterize    Alexandrov tangent maps using the classification of homogeneous harmonic maps into a conical surface (cf.~Kuwert \cite{kuwert}).  From this, we deduce that  non-degenerate harmonic maps are discrete and open.  Thus, it follows by V\"ais\"al\"a's classical result that  proper, non-degenerate harmonic maps (which include  proper, almost conformal harmonic maps) are local homeomorphisms away from a set of topological dimensional zero, the branch set.  We further improve this result and prove that the branch set is discrete by an application of the  order function and using the structure of Alexandrov tangent maps.  Consequently, we conclude that proper, non-degenerate harmonic maps  are branched coverings.  The following is an outline of  {\sc Part I}:
  
  \begin{itemize}
  \item[\S2] {\sc Preliminaries}.
  We recall the definitions of CAT$(\kappa)$ spaces and tangent cones.  Furthermore, we recall  the Korevaar-Schoen Sobolev spaces into metric spaces, including the notions of harmonic maps, pullback metrics and almost conformal maps, and explore the relationship of the Korevaar-Schoen energy density functions,  metric differential and the Jacobian. 
  \item[\S3] {\sc Tangent maps}.  We recall the notion of tangent maps  from \cite{gromov-schoen} and Alexandrov tangent maps from \cite{daskal-meseFPT}.  The main goal is to prove Theorem~\ref{tangent_maps_are_cones}, which asserts  that an Alexandrov tangent map associated to a harmonic map into a compact CAT($\kappa$) space with the geodesic extendability condition is a  homogeneous harmonic map. Note that a  CAT($\kappa$) surface satisfies the geodesic extendability condition.
  \item[\S4]  {\sc Non-degenerate harmonic maps.}
 This section contains the technical results needed to show that proper, non-degenerate harmonic maps are in fact branched covers. It also contains the proofs of Theorem~\ref{Main1} (end of Section \ref{locdeg}) and Theorem~\ref{sphere_cor} (end of Section \ref{Hksection}).

First, we define the notion of non-degenerate harmonic maps between surfaces and show that such maps are discrete (cf.~Lemma~\ref{discrete}) and open (cf.~Proposition~\ref{omt}). By further  analysis, we demonstrate that the branch set is discrete and  every such map is a branched cover (cf.~Theorem~\ref{itsbranched}).  In Proposition~\ref{Main2}, we relate the stretch constant of an Alexandrov tangent map to  the quasiconformal constant of the harmonic maps. Finally, we show that every almost conformal harmonic map is non-degenerate (cf.~Lemma~\ref{nd_lemma}).
  \end{itemize}

{\sc Part II.}  We apply the results of Part I to find a  harmonic conformal parameterization of a locally CAT($\kappa$) sphere $(\Sp^2,d)$. We 
start by using the curvature assumption to construct a finite energy map. We then  employ Corollary~\ref{cor:bubbling}, i.e.~the generalization of the Sacks-Uhlenbeck ``bubbling",  asserting 
the existence of a  harmonic map $u:\Sp^2 \rightarrow (\Sp^2,d)$. Although $u$ may not be a homeomorphism (it may not  be a  degree 1 map), Theorem~\ref{sphere_cor} tells us that it is a 
conformal branched covering. The map $u$ thus defines a complex structure on $\Sp^2$ and a map (which we call $\text{id}$ since it is essentially the identity map) from the quotient space defined by the branched 
cover. We study the relationship between the energy of an almost conformal homeomorphism and the area of its image to show that all such maps satisfy the expected area and energy equality and moreover are locally energy minimizing. Thus,  $\text{id}$ 
is, at least away from the branch points of $u$,   an almost conformal harmonic homeomorphism. Applying the
 removable singularities theorem of \cite{paper2}, we demonstrate that $\text{id}$ extends to an almost conformal harmonic homeomorphism on all of $\Sp^2$. The 1-quasiconformality follows from Theorem~\ref{sphere_cor}. Finally, we prove that the map is unique up to a M\"obius transformation.  The following is an outline of  {\sc Part II}.
\begin{itemize}
\item[\S5]{\sc Proof of Theorem~\ref{Main3}}.
\begin{itemize}
\item[\S5.1] We prove the uniqueness statement in Proposition~\ref{unique_mobius}:    If a conformal harmonic homeomorphism exists, then it is unique up to a M\"obius  transformation of $\Sp^2$. 
\item [\S5.2]  We explore the relationship between energy of a map and area  of its image.  In particular, for monotone maps  into a CAT($\kappa$) surface, being almost conformal is equivalent to energy being equal to twice the 2-dimensional Hausdorff measure of the image (cf.~Lemma~\ref{coarea}). 
\item[\S5.3]  We complete the proof of  Theorem~\ref{Main3} by first proving the existence of a finite energy map and then appealing to Corollary~\ref{cor:bubbling} and Theorem~\ref {sphere_cor} to find an almost conformal harmonic branched cover. From there the proof follows as outlined above.
\end{itemize}
\end{itemize}
 
\noindent {\sc Acknowledgements:  }The authors would like to thank A.~Lytchak and M.~Romney for their interest in this work and useful conversations. \section{Preliminaries}\label{prelims}

\subsection{CAT($\kappa$) space}  We review the notion of a CAT($\kappa$) space.   Intuitively, triangles in a CAT($\kappa$) space are ``slimmer" than corresponding ``model triangles" in a standard space of constant curvature $\kappa$. These spaces generalize Riemannian manifolds of sectional curvature bounded from above by $\kappa$. 

Let $\kappa>0$.  A metric space $(X,d)$ is called a \emph{$\frac{\pi}{\sqrt{\kappa}}$-geodesic space} if for each $P, Q \in X$ such that $d(P,Q) < \frac{\pi}{\sqrt{\kappa}}$, there exists a curve $\gamma_{PQ}$ such that the length of $\gamma_{PQ}$ is exactly $d(P, Q)$.  We call $\gamma_{PQ}$ a \emph{geodesic} between $P$ and $Q$.   We denote the geodesic ball of radius $r>0$ centered at $P \in X$ by $\BB^X_r(P)$ (or $\BB^{(X,d)}_r(P)$ whenever more than one distance function is defined on $X$).  We may drop the superscript $X$ when the context is clear.  Given a $\frac{\pi}{\sqrt{\kappa}}$-geodesic space $(X,d)$, a geodesic $\gamma_{PQ}$ with $d(P,Q)<\frac{\pi}{\sqrt{\kappa}}$ and  $t \in [0,1]$, let 
\[
P_t=(1-t) P + t Q
\]
 denote the point on $\gamma_{PQ}$ at distance $t d(P,Q)$ from $P$. Given  three points $P,Q,R \in X$ such that $d(P,Q)+d(Q,R)+d(R,S) < \frac{2\pi}{\sqrt{\kappa}}$, the \emph{geodesic triangle} $\triangle PQR$ is the triangle in $X$ with sides given by the geodesics $\gamma_{PQ}, \gamma_{QR}, \gamma_{RS}$.  

Let $\Sp^2$ be the standard unit sphere and let  $\Sp^2_\kappa$ denote the scaled version of $\Sp^2$ with  Gauss curvature $\kappa$. Let $\tilde d$ be the induced distance function on $\Sp^2_\kappa$.  
A \emph{comparison triangle} for the geodesic triangle $\triangle PQR$ in a $\frac{\pi}{\sqrt{\kappa}}$-geodesic space is a geodesic triangle $\triangle \tilde{P}\tilde{Q}\tilde{R}$ on  $\Sp^2_\kappa$ such that $d(P,Q)=\tilde d(\tilde{P},\tilde{Q})$, $d(Q,R)=\tilde d(\tilde{Q},\tilde{R})$  and $d(R,P)=\tilde d(\tilde{R},\tilde{P})$.  

\begin{definition}
Let $(X,d)$ be a metric space. Then $X$ is a CAT($\kappa$) space if it is a complete $\frac{\pi}{\sqrt{\kappa}}$-geodesic space satisfying the following:
If $\triangle PQR$ is a geodesic triangle with perimeter less than $\frac{2\pi}{\sqrt{\kappa}}$  and  $\triangle \tilde{P}\tilde{Q}\tilde{R}$ in $\Sp^2_{\kappa}$ is a comparison triangle, then,  for $t, \tau \in [0,1]$,
\begin{equation} \label{cat}
d(P_t,R_\tau) \leq \tilde d(\tilde{P}_t,\tilde{R}_\tau)
\end{equation}
where 
\begin{eqnarray*}
P_t=(1-t)P+tQ,  & &  R_\tau=(1-\tau)R+\tau Q,\\
\tilde{P}_t=(1-t)\tilde{P}+t\tilde{Q}, & &  \tilde{R}_\tau=(1-\tau)\tilde{R}+\tau\tilde{Q}.
\end{eqnarray*}
A complete geodesic space $X$ is said to be \emph{locally}   CAT($\kappa$) if, for every point $P$ of $X$, there exists $r>0$ sufficiently small such that $\overline{\BB^X_r(P)}$  is a  CAT($\kappa$) space.
\end{definition}

\begin{remark}
 A CAT(0) space (or an NPC space) is  a complete geodesic space satisfying inequality  (\ref{cat}) with $\Sp^2_\kappa$ replaced by $\R^2$ and with no perimeter restriction.  
\end{remark}
We recall the notion of angles and tangent spaces in a locally CAT($\kappa$) space $(X,d)$.  
Fix $q_0 \in X$, and let ${\mathcal G}_{q_0}$ be the set of all geodesics emanating from $q_0$.  For  $\gamma  \in {\mathcal G}_{q_0}$ (resp. $\hat{\gamma} \in {\mathcal G}_{q_0})$ and $q_1 \in \gamma$ (resp. ~$q_2 \in \hat{\gamma})$ with $q_1 \neq q_0$ (resp.~$q_2 \neq q_0$) sufficiently close to $q_0$,  the \emph{comparison angle }$\widetilde{\angle}_{q_0}(q_2,q_1)$ is the angle  at the point corresponding to $q_0$ of the comparison triangle to $\triangle q_0q_1q_2$ in ${\mathbb S}^2_\kappa$. 
By the CAT($\kappa$) assumption, the function
\begin{equation} \label{anglemonotone}
t \mapsto \widetilde{\angle}_{q_0}(Q(t),P(t))
\mbox{ is non-decreasing} 
\end{equation}
where $Q(t)$ (resp.~$P(t)$) is a constant speed parameterization of  $\gamma$ (resp.~ $\hat{\gamma}$) with $Q(0)=q_0$ (resp.~$P(0)=q_0$). Thus, the limit
\[
\angle (\gamma,\hat{\gamma}):=\lim_{t \rightarrow 0} \widetilde{\angle}_{q_0}(Q(t),P(t))
\]
exists and this is the \emph{angle} between the geodesics $\gamma$ and $\hat{\gamma}$.

 Define an equivalence relation in ${\mathcal G}_{q_0}$ by letting
\[
\gamma_1 \sim \gamma_2 \ \mbox{ \ if and only if \ } \ \angle (\gamma_1,\gamma_2)=0.
\]
The \emph{space of directions} ${\mathcal E}_{q_0}$ is the completion of the metric space of equivalence classes $[\gamma]$ of ${\mathcal G}_{q_0}$  with distance function $\Theta(\cdot,\cdot)$ defined by 
\[
\Theta([\gamma_1],[\gamma_2]) =\angle(\gamma_1,\gamma_2).
\]
The \emph{(Alexandrov) tangent cone of $(X,d)$ at $q_0$} is the CAT(0) space 
\[
T_{q_0}X= [0,\infty) \times {\mathcal E}_{q_0}/\sim',
\]
where $\sim'$ identifies all points of the form $(0,[\gamma])$ as the vertex ${\mathcal O}$, along with the distance function given by
\[
\delta^{{ 2}}((\rho_1, [\gamma_1]), (\rho_2,[\gamma_2])) = \rho_1^2+\rho_2^2-2\rho_1\rho_2 \cos \Theta([\gamma_1],[\gamma_2]). 
\]
For a sufficiently small neighborhood $\NN$ of $q_0$, there is a natural projection map 
\begin{equation} \label{logmap}
\log_{q_0}: \NN \rightarrow T_{q_0}X
\end{equation}
\[
\log_{q_0}(q)=(d(q,q_0), [\gamma_q])
\]
where $\gamma_q$ is a geodesic ray emanating from $q_0$ that goes through $q$.

The main interest in this paper is CAT($\kappa$) surfaces and spheres.

\begin{definition}  \label{lcs}
A  CAT($\kappa$) space (resp.~locally CAT($\kappa$) space) $(X,d)$ is a \emph{CAT($\kappa$) manifold} (resp.~\emph{locally CAT($\kappa$) manifold}) if, for  every point $p \in X$, there exists $r>0$ sufficiently small such that $\BB^X_r(p)$ is homeomorphic to a unit ball in  $\R^n$. We will say a CAT($\kappa$) manifold (resp.~locally CAT($\kappa$) manifold) $(X,d)$ is  
a  \emph{CAT($\kappa$) surface} (resp.~\emph{locally CAT($\kappa$) surface}) if $n=2$. Finally, a \emph{CAT($\kappa$) sphere} (resp.~\emph{locally CAT($\kappa$) sphere}) is a CAT($\kappa$) surface (resp.~locally CAT($\kappa$) surface) which is homeomorphic to $\Sp^2$.
\end{definition}

\begin{remark} \label{3properties}
If $(X,d)$ is a locally CAT($\kappa$) manifold, then
for each $q_0 \in X$ there exists $r>0$ sufficiently small such that the closed geodesic ball $\overline{\BB_r(q_0)}$ is a CAT$(\kappa$) space and the following properties are satisfied:  
\begin{itemize}
\item[(i)] (Uniqueness of geodesics) There exists a unique geodesic between every pair of  points in $\BB_r(q_0)$ and $\BB_\epsilon(q_0)$ is convex for every $\epsilon \in (0,r]$ (cf.~\cite[II.1.4]{bridson-haefliger}).
\item[(ii)] (Continuity of angles) For geodesics  $\gamma_p$ and $\gamma_q$ in $\BB_r(q_0)$, from $q_0$ to $p$ and $q$ respectively,  the function $(p,q) \mapsto \angle (\gamma_p,\gamma_q)$ is continuous (cf.~\cite[II.3.3]{bridson-haefliger}).
\item[(iii)] (Geodesic extendability)  Every geodesic from $q_0$ to a point in $\BB_r(q_0)$  can be extended to a geodesic from $q_0$ to a point in $\partial \BB_r(q_0)$  (cf.~\cite[II.5.12]{bridson-haefliger}).  
 \end{itemize}  
 \end{remark}

We use the following proposition  in our analysis of tangent maps for harmonic maps into CAT($\kappa$) surfaces. Since this is already known to the experts (e.g.~\cite{reshetnyak0}), we will only state it here and defer its proof to Appendix~\ref{appB}.

\begin{proposition} \label{coneovercurve}
If $(S,d)$ is a locally CAT($\kappa$) surface,  then the Alexandrov tangent cone $T_{q_0}S$ of $S$ at $q_0 \in S$ is a metric cone over a finite length simple closed  curve.  More precisely, the space of directions $\EE_{q_0}$ is isometric to a  finite length simple closed curve.
\end{proposition}

\subsection{Korevaar-Schoen energy and harmonic maps} We refer the reader to \cite{korevaar-schoen1} for  details and background on the notion of  finite energy (or $W^{1,2}$) maps into metric spaces that we briefly summarize here.  Note that we will be restricting the general theory of \cite{korevaar-schoen1} to case when the domain dimension is 2.   
\begin{definition}
Let $\Sigma$ be a Riemann surface.  A \emph{holomorphic disk} $\D \subset \Sigma$  is a coordinate neighborhood of $\Sigma$  identified as a unit disk in the complex plane $\C$ by the conformal coordinate $z=x+iy$.  We will say a holomorphic disk $\D$ \emph{is centered at $p$} if  $p \in \Sigma$ is identified with 0.   Furthermore, we denote for $r \in (0,1)$,
\[
\D_r=\{z \in \D:  |z|<r\}.
\]
\end{definition}

Let $\Sigma$ be a Riemann surface and $(X,d)$ be a complete metric space. 
The Sobolev space  $ W^{1,2}(\Sigma,X) \subset L^2(\Sigma,X)$ is the space of finite energy maps $u:\Sigma \rightarrow (X,d)$. We recall that (because we restrict to the case when the domain dimension is 2) the energy of a map depends only on the conformal class of $\Sigma$ (and not on the metric defined on $\Sigma$).  

For $f \in W^{1,2}(\Sigma,X)$, a holomorphic disk $\D$ and $\Gamma(T\D)$  the space of Lipschitz vector fields on $\D$,  we denote the \emph{directional energy density function} for $Z \in \Gamma(T\D)$ (cf.~\cite[Section 1.7ff.]{korevaar-schoen1}) and \emph{energy density function} (cf.~\cite[Section 1.10ff.]{korevaar-schoen1}) of $f$ on $\D$ by  
\[
|f_*(Z)|^2  \ \mbox{ and } \ |\nabla f|^2.
\]
Let $\{\partial_x, \partial_y\}$ be the standard orthonormal basis on $\D$.  
For a.e.~$z \in \D$, we have (cf.~\cite[(1.10v)]{korevaar-schoen1})
\begin{equation} \label{110v}
\frac 12{|\nabla f|^2(z)}= \frac{1}{2\pi} \int_{\omega \in \Sp^1} |f_*(\omega)|^2 (z) d\theta(\omega)
\end{equation}
where $\omega \in \Gamma(T\D)$ is given by $\omega=\cos \theta \cdot \partial_x+ \sin \theta \cdot \partial_y$ for a fixed constant  $\theta \in [0,2\pi)$.  Note that we have identified the set of such $\omega$'s with $\Sp^1=\{e^{i\theta}\}$ in the obvious way.  

The measure  $|\nabla f|^2\, dxdy$ is defined independently of the local holomorphic coordinates. The total energy of $f\in W^{1,2}(\Sigma,X)$ is given by
\[
	^{d}E^f =\int_{{\Sigma}}|\nabla f|^2 dxdy.
\]  
Given a subdomain $\Omega$ of $\Sigma$, we denote the energy of $f$ in $\Omega$ by
\[
^{d}E^f[\Omega]=
\int_{{\Omega}}|\nabla f|^2 dxdy.
\]  
Given $h \in W^{1,2}(\D,X)$, we define
\[
W^{1,2}_h(\D,X)=\{ f \in W^{1,2}(\D,X): Tr(h)=Tr(f)\}
\]
where $Tr(f)$ denotes the trace map of $f \in W^{1,2}(\D,X)$.  

\begin{definition}
\label{def:energyminizing}
We say $u \in W^{1,2}(\D,X)$ is \emph{energy minimizing} if there exists $P \in X$ and $\rho >0$ such that $u(\D) \subset \BB^X_{\rho}(P)$ and  
$u$ minimizes energy among all maps in $W^{1,2}_u(\D,\overline{\BB^X_\rho(P)})$. 
\end{definition}
\begin{definition}
\label{def:harmonic}
We say that a map $u \in W^{1,2}(\Sigma,X)$ is a \emph{harmonic map} if it is locally energy minimizing; more precisely, for every $p \in \Sigma$, there exists a holomorphic disk $\D \subset \Sigma$ centered at $p$ such that $u\big|_{\D}$ is energy minimizing.
\end{definition}

\begin{theorem}[\cite{serbinowski}, \cite{BFHMSZ}] \label{lip}
If $u:\D \rightarrow (X,d)$ is an energy minimizing map from a holomorphic disk $\D \subset \Sigma$ into a  CAT($\kappa$) space  $X$, then $u$ is locally Lipschitz continuous.
The Lipschitz constant depends only on $^{d}E^u$, the metric on $\Sigma$, and the distance to $\partial \D$. \end{theorem} 

\begin{lemma} \label{contbdry}
If $u$ is as in Theorem~\ref{lip} and $Tr(u) \in C^0(\partial \D)$, then 
\[
\bar u=\left\{
\begin{array}{ll}
u & \mbox{in }\D
\\
Tr(u) & \mbox{in } \partial \D
\end{array}
\right.
\]
 is continuous in $\bar \D$.  
\end{lemma}

\begin{proof}
By Theorem~\ref{lip}, it is sufficient to prove the continuity of $\bar u$ at $z_0 \in \partial \D$.  Let $\epsilon>0$ be sufficiently small such that the nearest point projection map onto any closed geodesic ball of radius $\epsilon$ is distance non-increasing in the geodesically convex, CAT($\kappa$) space $\NN \subset X$ (cf.~\cite[II.2.6(2)]{bridson-haefliger}).  
By the continuity of $Tr(u)$, there exists  $\delta_1>0$ sufficiently small such that 
\[
\bar u(\partial \D \cap \D_{\delta_1}(z_0)) \subset \BB_{\frac{\epsilon}{2}}(\bar u(z_0)).
\]  By the Courant-Lebesque Lemma, there  exists  $\delta \in (0,\delta_1)$,  $r \in (\delta^2, \delta)$ and $Q \in X$ satisfying
\[
\bar u(\partial \D_{r}(z_0) \cap \bar \D) \subset \BB_{\frac{\epsilon}{2}}(Q).
\]   
Since 
\[
\bar u(\zeta) \in \BB_{\frac{\epsilon}{2}}(\bar u(z_0)) \cap \BB_{\frac{\epsilon}{2}}(Q) \mbox{ for } \zeta  \in \partial \D \cap \partial \D_r(z_0),
\]
we have
\[
\bar u(\partial (\D_{r}(z_0) \cap \bar \D)) \subset \BB_\epsilon(\bar u(z_0)).
\]
By the energy minimizing property of $u$ and since the nearest point projection map into $\overline{\BB_\epsilon(\bar u(z_0))}$ does not increase energy,
\[
\bar u(\D_{r}(z_0) \cap \bar \D) \subset \overline{\BB_\epsilon(\bar u(z_0))}.
\]
\end{proof}

\subsection{Almost conformal maps} 
Let  $u: \Sigma \rightarrow (X,d)$ be a harmonic map  from a Riemann surface into a locally CAT($\kappa$) space.  Recall the construction in \cite{korevaar-schoen1}, \cite{mese} and \cite{BFHMSZ} of a continuous, symmetric, bilinear, non-negative $L^1$ tensorial operator associated  with $u$, 
\begin{equation} \label{pi}
\pi: \Gamma(T\Sigma) \times \Gamma(T\Sigma) \rightarrow L^1(\Sigma)
\end{equation}defined by
\[
\pi(Z,W) := \frac{1}{4} |u_*(Z+W)|^2-\frac{1}{4}|u_*(Z-W)|^2.
\]
This generalizes the notion of the pullback metric for maps into a Riemannian manifold, and hence we shall refer to $\pi$ also as the \emph{pullback metric} for $u$. The energy of $u$ can be written as
\[
^dE^u=\int_{\Sigma} \pi \left( \frac{\partial}{\partial x},\frac{\partial}{\partial x} \right)+ \pi \left( \frac{\partial}{\partial y},\frac{\partial}{\partial y} \right) dxdy.
\]
\begin{definition} \label{area_def}
The \emph{area of $u$} is  
\[
^dA^u =  \int_{\Sigma} \sqrt{\pi \left( \frac{\partial}{\partial x},\frac{\partial}{\partial x} \right) \pi \left( \frac{\partial}{\partial y},\frac{\partial}{\partial y} \right) -\left( \pi \left( \frac{\partial}{\partial x},\frac{\partial}{\partial y}\right)\right)^2} dxdy.
\]
\end{definition}

\begin{lemma} \label{holomorphic}
Let $u: \Sigma \rightarrow (X,d)$ be a harmonic map from a Riemann surface into a locally CAT($\kappa$) space. The \emph{Hopf differential} $\Phi=\phi\, dz^2$ of $u$, defined in a holomorphic disk $\D \subset \Sigma$ where 
\[
     \phi(z):=\left[ \pi \left( \frac{\partial}{\partial x},\frac{\partial}{\partial x} \right)  
                        - \pi \left( \frac{\partial}{\partial y},\frac{\partial}{\partial y} \right) 
                      -2i \pi \left( \frac{\partial}{\partial x},\frac{\partial}{\partial y}\right) \right], \, 
\]
is holomorphic. 
\end{lemma}

\begin{proof} 
Let $\zeta$ be a smooth function with compact support in a holomorphic disk $\D \subset \Sigma$.  For $\epsilon>0$ sufficiently small and $t \in (-\epsilon,\epsilon)$, consider the diffeomorphism $F_t:\Sigma \rightarrow \Sigma$ given in $\D$ by $F_t(z) = (1 + t \zeta(z))z$   and  $F_t=\ $identity outside of $\D$.  Using the domain variation $t \mapsto F_t$, the assertion  follows from following the argument of   \cite[Lemma 1.1]{schoen} (\cite[Chapter 3]{jost}), where the change of variables is justified by \cite[Theorem 2.3.2]{korevaar-schoen1}. 
\end{proof}

\begin{definition} \label{def:conformal}
The map $u \in W^{1,2}(\Sigma,X)$ is said to be {\em almost conformal} if, for any holomorphic disk $\D \subset \Sigma$ and  a.e.~$z \in \D$,
\[
\pi \left( \frac{\partial}{\partial x},\frac{\partial}{\partial x} \right)  
                        = \pi \left( \frac{\partial}{\partial y},\frac{\partial}{\partial y} \right) \ \ \mbox{ and }  \ \ 
                      \pi \left( \frac{\partial}{\partial x},\frac{\partial}{\partial y}\right)  =0.
\]  
 \end{definition}

\begin{lemma} \label{holomorphic0}
A harmonic map $u:\Sp^2 \rightarrow (X,d)$ from the standard 2-sphere to a locally CAT($\kappa$) space is almost conformal.
\end{lemma}

\begin{proof}
 The only holomorphic quadratic differential on $\Sp^2$ is identically equal to 0.  
 \end{proof}

\begin{remark} \label{AE}
From the Cauchy-Schwarz inequality, 
\[
{^d}A^u \leq {^dE^u}/2
\]
with equality if and only if $u$ is an almost conformal map.
\end{remark}
\begin{definition}

Let  $u:\Sigma \rightarrow \Sigma$ be an almost conformal map and $\D \subset \Sigma$ be a holomorphic disk.    The function $\lambda_u: \D \rightarrow [0,\infty)$ defined by
\[
\lambda_u= \pi \left( \frac{\partial}{\partial x},\frac{\partial}{\partial x} \right)   = |\nabla u|^2/2
\]
is called the \emph{conformal factor} of $u$ in $\D$.  Note that the pullback metric of $u$ in $\D$ is given by $\lambda_u (dx^2+dy^2)$.  

\end{definition}

 \begin{theorem}[\cite{mese}, Theorem~4.1, Theorem~5.1 and Lemma~5.2] \label{lambdadifeq}
If $u:\Sigma \rightarrow (X,d)$ is an almost conformal harmonic map from a Riemann surface to a locally CAT($\kappa$) space and $\D$ a holomorphic disk, then the conformal factor $\lambda_u$ of $u$ in $\D$ satisfies
\begin{itemize}
\item $\lambda_u$ is locally bounded.
\item $\lambda_u \in W^{1,2}_{loc}(\D)$.
\item $\triangle \lambda_u \geq -2 \kappa \lambda_u^2$ weakly. 
\item $\triangle \log \lambda_u \geq -2 \kappa \lambda_u$ weakly.
\end{itemize}
\end{theorem}

\subsection{Existence of harmonic maps} As mentioned in the introduction, Sacks and Uhlenbeck \cite{sacks-uhlenbeck}  discovered a ``bubbling phenomena" for harmonic maps from surfaces. The paper  \cite{BFHMSZ} considers an analogous  result  when the target space is a compact locally CAT($\kappa$) space.

\begin{theorem}[\cite{paper2}] \label{bubbling}
Let $\Sigma$ be a compact Riemann surface, $(X,d)$ a compact locally CAT($\kappa$) space and $\varphi \in C^0 \cap W^{1,2}(\Sigma,X)$. Then either there exists a harmonic map $u:\Sigma \rightarrow (X,d)$ homotopic to $\varphi$ or an almost conformal harmonic map $v:\mathbb S^2 \rightarrow (X,d)$.
\end{theorem}

In the case when $\Sigma$ is the standard 2-sphere $\Sp^2$, Theorem~\ref{bubbling} implies the following.

\begin{corollary} \label{cor:bubbling}
 If there exists a  continuous, finite energy map $h$ from the standard 2-sphere  $\Sp^2$ into a compact locally CAT($\kappa$) space $(X,d)$ then there exists an almost conformal harmonic map $u:\Sp^2 \rightarrow (X,d)$.
\end{corollary}

\subsection{Metric differential and Jacobian}
Throughout this subsection, $(X,d)$ will denote a complete metric space.  The following definitions are  given in \cite{kirchheim} and \cite[Definition 7.9]{karmanova} respectively.

\begin{definition}
Let $f:\D \rightarrow X$ and $z_0 \in \D$.   If there exists a seminorm $s:\C \rightarrow [0,\infty)$ satisfying
\[
 \lim_{z \rightarrow z_0} \frac{s(z-z_0) - d(f(z), f(z_0))}{|z-z_0|}=0,
\]
then $\text{MD}(f,z_0):=s$ is said to be the  \emph{metric differential} of $f$  at $z_0 \in \D$.
\end{definition}

\begin{remark}
Kirchheim \cite[Theorem 2]{kirchheim} proved that if $f:\D \rightarrow X$ is a Lipschitz map, then MD$(f,z_0)$ exists for a.e.~$z_0 \in \D$.
\end{remark}

\begin{definition}
Let $f:\D \rightarrow X$ and $z_0 \in \D$.   If there exists a seminorm $s:\C \rightarrow [0,\infty)$ satisfying
\[
\text{ap} \lim_{z \rightarrow z_0} \frac{s(z-z_0) - d(f(z), f(z_0))}{|z-z_0|}=0,
\]
then  $\text{MD}_{ap}(f,z_0):=s$ is said to be the \emph{approximate metric differential} of $f$  at $z_0 \in \D$.
Recall that a function $\varphi: \D \rightarrow \R$  has an \emph{approximate limit} $L=\text{ap }\lim_{z \rightarrow z_0}\varphi(z)$  at $z_0$  if there exists a set $A$ that has density 1 at $z_0$ such that if $z_n$ is a sequence in $A$ and $z_n \rightarrow z_0$,  then $\varphi(z_n) \rightarrow L$.  
\end{definition} 

\begin{remark}
Let $f, \hat f:\D \rightarrow X$ and $A \subset \D$ be a measurable set such that $f=\hat f$ in $A$.   If $z_0 \in A$ is a  density 1 point of $A$ such that $\text{MD}(\hat f, z_0)$ exists, then $\text{MD}_{ap}(f,z_0)$ exists and 
$\text{MD}(\hat f, z_0)=\text{MD}_{ap}(f,z_0)$.
\end{remark}
The following definition can be found in  \cite[Definition 5]{kirchheim} or  \cite[Theorem 7.10]{karmanova}.
\begin{definition}
The Jacobian of a map $f:\D \rightarrow X$ at $z_0 \in \D$ is defined as
 \[
 \JJ_f(z_0) =
  \left( \frac{1}{2\pi} \displaystyle{ \int_{\omega \in \Sp^1} \big( \text{MD}_{ap}(f,z_0)(\omega)  \big)^{-2}d\HH^1(\omega)} \right)^{-1} 
  \]
  whenever $\text{MD}_{ap}(f,z_0)(\omega) \neq 0$ for a.e.~$\omega \in \Sp^1$.  Otherwise, define $\JJ_f(z_0)=0$.  
\end{definition}

The following lemma relates  the metric differential of a finite energy map to its  directional energy density function. We defer the proof to Appendix \ref{appC}.

\begin{lemma} \label{KSjac}
If $f \in W^{1,2}(\D,X)$, then for a.e.~$z_0 \in \D$
\[
 \mathrm{MD}_{ap}(f,z_0)(\omega)=\left|f_*(\omega)\right|(z_0), \ \ \text{a.e.}~\omega \in \Sp^1
 \]
In particular,  the Jacobian of  $f$ at a.e.~$z_0 \in \D$  with  $\text{MD}_{ap}(f,z_0)(\omega) \neq 0$ for a.e.~$\omega \in \Sp^1$ is
  \[
 \JJ_f(z_0) =
 \left(  \frac{1}{2\pi} \displaystyle{\int_{\omega \in \Sp^1}  |f_*(\omega)|^{-2}(z_0) d\HH^1(\omega) }\right)^{-1}
 \]
whenever  $|f_*(\omega)|^2(z_0) \neq 0$ a.e.~$\omega \in \Sp^1$.  Otherwise,  $\JJ_f(z_0)=0$.
\end{lemma}

\section{Tangent maps}\label{TC_section}
The goal of this section is to prove Theorem \ref{tangent_maps_are_cones} which shows that for harmonic maps into CAT($\kappa$) manifolds, an Alexandrov tangent map of a harmonic map $u$ is itself a tangent map of $u$. 
(Since harmonic maps are continuous, all the ``local'' results in this section stated for  CAT($\kappa$) spaces remain valid after replacing by locally CAT($\kappa$).)

\subsection{Construction of tangent maps}

 Let  
$
u:\D \rightarrow (X,d)$ be a harmonic map to a CAT($\kappa$) space, $p_0 \in \Sigma$ and $\D \subset \Sigma$ be a holomorphic disk centered at $p_0$.  
We will now summarize the construction of a \emph{tangent map} of $u$.  (For more details, we  refer the reader to \cite{korevaar-schoen2}  where the notion of convergence for a sequence of maps into different NPC  spaces first appears, and also \cite{{BFHMSZ}} where this notion is generalized from NPC to CAT($\kappa$) spaces.)

For  $\sigma>0$ sufficiently small, let 
\begin{equation}\label{mu}
\mu_{\sigma}:=\sqrt{\frac{\displaystyle{\int_{\partial \D_{\sigma}}d^2(u,u(0))\,d\theta}}{\sigma}}.
\end{equation}
We construct a  CAT($\mu_{\sigma}^2\kappa$) space $(X,d_{\sigma})$  by endowing $X$ with  a distance function  
\begin{equation} \label{dsigma}
d_{\sigma}(q,q')=\mu_{\sigma}^{-1} d(q,q').
\end{equation}
A \emph{blow up map}   of $u$ at $p_0$  is  
\[
u_{\sigma}:\D \rightarrow (X,d_{\sigma}),
\ \ 
u_{\sigma}(x)=u(\sigma x).
\]

By \cite[Proposition 6.5 and Section 8]{{BFHMSZ}}, 
\begin{equation} \label{orderlimitexists}
\lim_{\sigma \rightarrow 0}  \frac{{\displaystyle{\int_{\D} |\nabla u_\sigma|^2}\,dxdy } }{\displaystyle{\int_{\partial \D } d^2_\sigma(u_\sigma,u_\sigma(0))\, d\theta}} = \lim_{\sigma \rightarrow 0}  \frac{\sigma {\displaystyle{\int_{\D_\sigma} |\nabla u|^2}\,dxdy } }{\displaystyle{\int_{\partial \D_\sigma} d^2(u,u(0))\, d\theta}}=:\mathrm{ord}^u(0) \mbox{ exists and } \mathrm{ord}^u(0) \geq 1.
\end{equation}
The normalization by $\mu_\sigma$ implies that 
\begin{equation} \label{bmnorm}
\int_{\partial \D} d_\sigma^2(u_\sigma, u_\sigma(0)) \, d\theta=1.
\end{equation}
Thus, the energy of $u_\sigma$ is uniformly bounded, and 
by Theorem \ref{lip}, $\{u_{\sigma}\}$ is uniformly Lipschitz continuous in $\D_r$ for any $r \in (0,1)$.

We now inductively define maps 
$
\{u_{\sigma,i}\}$ 
and pullback pseudodistance functions
$\{\rho_{{\sigma,i}}\}
$
 as follows: First, we let 
 \begin{eqnarray*}
 \Omega_{0} & = & \D.
 \end{eqnarray*}
 Having defined $\Omega_{i-1}$, we inductively define
 \begin{eqnarray*}
\Omega_i & = &  \Omega_{i-1} \times \Omega_{i-1} \times [0,1].
 \end{eqnarray*}
    Identify $\Omega_i \subset \Omega_{i+1}$ via the inclusion $x \mapsto (x,x,0)$ and set
\[
\Omega_{\infty} =  \bigcup_{i=0}^{\infty} \Omega_i.
\] Next, let 
\[
u_{\sigma,0}=u_{\sigma}:\Omega_0 \rightarrow (X,d_{\sigma}).
\]
Having defined the map 
$
u_{\sigma,i-1}:\Omega_{i-1} \rightarrow (X,d_{\sigma})$, we define  
\[
u_{\sigma,i}: \Omega_i \rightarrow (X,d_{\sigma}), \ \ u_{\sigma,i}(x,y,t)=\gamma(t)
\]
 where $\gamma:[0,1] \rightarrow (X,d_{\sigma})$ is the constant speed parameterization of the unique geodesic from $u_{\sigma,i-1}(x)=\gamma(0)$ to $u_{\sigma,i-1}(y)=\gamma(1)$. Let    
 \[
 \rho_{\sigma,i}: \Omega_{i} \times \Omega_i \rightarrow [0,\infty), \ \  \ \rho_{\sigma,i}(x,y) =d_{\sigma}(u_{\sigma,i}(x), u_{\sigma,i}(y)).
 \]  
Finally, we define 
\[
\rho_{\sigma, \infty}: \Omega_\infty \times \Omega_\infty \rightarrow [0,\infty), \ \ \rho_{\sigma, \infty}\big|_{\Omega_i}=\rho_{\sigma,i}.
\]  
We define an equivalence relation $\sim_{\rho_{\sigma,\infty}}$ by setting
\[
x \sim_{\rho_{\sigma,\infty}}y \Leftrightarrow \rho_{\sigma, \infty}(x,y)=0.
\]
Then $\rho_{\sigma,\infty}$ is a distance function on $\Omega_{\infty}/\sim_{\rho_{\sigma, \infty}}$ and let $X_\infty:=\overline{\Omega_{\infty}/\sim_{\rho_{\sigma, \infty}}}$ denote its metric completion.  We can isometrically identify 
\[
X_\infty:=\overline{\Omega_{\infty}/\sim_{\rho_{\sigma, \infty}}}
\approx \overline{Cvx(u_{\sigma}(\D))}.
\]

As explained in \cite{BFHMSZ},  there exists a sequence 
\begin{equation} \label{subseq}
\sigma_j \rightarrow 0
\mbox{ and } \rho_{*,i}:\Omega_i \times \Omega_i \rightarrow [0,\infty) \mbox{ for $i=0,1,\dots$}
\end{equation}
 such that the pullback pseudodistance functions  $\rho_{{\sigma_j,i}}$ converge locally uniformly to $\rho_{*,i}$ on each $\Omega_i$.  We thus obtain a pullback pseudodistance function 
\[
d_*:\Omega_{\infty} \times \Omega_{\infty} \rightarrow [0,\infty), \ \ \ d_*\big|_{\Omega_i}  =\rho_{*,i}.
\]  
We define an equivalence relation $\sim_*$ by setting
\[
x \sim_*y \Leftrightarrow d_*(x,y)=0
\]
and let  $\Omega_{\infty}/\sim_*$  denote the space of equivalent classes   $[ \cdot ]$.  
The metric completion  $X_*=\overline{\Omega_{\infty}/\sim_*}$  of $\Omega_{\infty}/\sim_*$ along with the distance function $d_*$ naturally defined on $X_*$ is an NPC space.  
Define
\begin{equation} \label{tm*}
u_*:  \D \rightarrow   (X_*,d_*),
\ \ \ \ u_*= \iota \circ \Pi
\end{equation}
where 
\[
 \iota: \D/\sim_* =\Omega_0/\sim_* \hookrightarrow X_*
\]
 is the inclusion map and 
\[
\Pi:\D \rightarrow \D/\sim_*, \ \ \ \Pi(z)=[z]
\]
is the natural projection map.

\begin{definition} \label{clu}
We say the sequence $f_j:\Omega_0 \rightarrow (X_j,d_j)$ 
 \emph{converges locally uniformly in the pullback sense} to $f: \Omega_0 \rightarrow (X_\infty,d_\infty)$
 if, for each $i$, the pullback pseudodistances of $f_{j,i}:  \Omega_i \rightarrow (X_j,d_j)$ 
 converge locally uniformly to the pullback pseudodistance of $f_i:\Omega_i \rightarrow (X_*,d_*)$.
\end{definition}

\begin{definition} \label{tm}
Any map $f:\D \rightarrow Y$ into an NPC space  is called a \emph{tangent map}  of a harmonic map $u:\D \rightarrow (X,d)$ 
  if there exists a sequence $\sigma_j \rightarrow 0$ such that $\{u_{\sigma_j}\}$ converges locally uniformly in the pullback sense to $f$.
\end{definition}    

\begin{remark}
For the sequence $\sigma_j \rightarrow 0$ as in (\ref{subseq}), the  sequence of blow up maps $\{u_{\sigma_j}\}$ converges  locally uniformly in the pullback sense to the map $u_*$ given by (\ref{tm*}) according to Definition~\ref{clu} and $u_*$ is a tangent map of $u:\D \rightarrow (X,d)$ according to Definition~\ref{tm}.
\end{remark}
As explained in \cite{gromov-schoen}, \cite{BFHMSZ},  and Appendix \ref{App_order}, a tangent map 
$u_*$ is a \emph{degree $\alpha$ homogeneous harmonic  map} where
\[
\alpha=\mathrm{ord}^u(0)=\mathrm{ord}^{u_*}(0) \geq 1
\]
 is the \emph{order} of $u$ at $0$ (cf.~(\ref{orderlimitexists})). 
  
This means that we can extend $u_*$ to $\overline \D$ by continuity and, for any $x \in \partial \D$,
\[
r \mapsto u_*(rx) \mbox{ parametrizes a geodesic in }X_*
\]
and
\begin{equation} \label{hg}
d_*(rx,0)=r^{\alpha}d_*(x,0), \ \ \forall r \in (0,1).
\end{equation} 

\subsection{Alexandrov tangent maps for maps with locally compact targets} \label{alex}

Next assume that $X$ is a locally compact CAT($\kappa$) space.  We review the notion of an Alexandrov tangent map introduced in \cite{daskal-meseFPT}.
Let  $q_0=u(p_0)$.  Let
\[
\log=\log_{q_0}:  \NN \subset X \rightarrow T_{q_0}X
\]
be the natural projection map (cf.~(\ref{logmap})) from a sufficiently small neighborhood $\NN$ of $q_0$.  Furthermore, let  $\{u_{\sigma_j}\}$ be a sequence of blow up maps at $p_0$ converging locally uniformly in the pullback sense to $u_*$.  We define
\[
\log_\sigma:(X,d_{\sigma}) \rightarrow (T_{q_0}X, \delta)
\]
 analogously to $\log$ (with $d$ replaced by $d_{\sigma}$).  Here we point out that the notion of a geodesic and of  $\angle$ are invariant under scaling of the distance function.  More specifically, if $\gamma$ is a  geodesic in $(X,d)$, then $\gamma$ is a geodesic in $(X,d_{\sigma})$.  Moreover, the value of $\angle (\gamma_1,\gamma_2)$ in $(X,d_{\sigma})$ is the same for any $\sigma>0$.   (On the other hand, $\widetilde{\angle}_{q_0}(q_1,q_2)$ depends on the distance function $d_{\sigma}$.)
 
 The map $\log_\sigma$ is a \emph{non-expansive map} (i.e.~distance non-increasing map). Thus, by Theorem \ref{lip},  $\{v_{\sigma} = \log_\sigma \circ u_{\sigma}\}$ is a sequence of maps into $T_{q_0}X$ with a uniform local Lipschitz bound.   Analogous to the construction of $u_{\sigma,i}$ and $\rho_{\sigma,i}$ from $u_{\sigma,0}=u_{\sigma}$, we start with $v_{\sigma,0}=v_{\sigma}$ and  inductively define
 \[
v_{\sigma,i}: \Omega_i \rightarrow (T_{q_0}X,\delta) 
\]
and
\[
 \hat \rho_{\sigma,i}: \Omega_{i} \times \Omega_i \rightarrow [0,\infty), \ \ \   \hat \rho_{\sigma,i}(x,y) =\delta(v_{\sigma,i}(x), v_{\sigma,i}(y)).
 \]  
 Since $X$ is locally compact, $T_{q_0}X$ is locally compact. Thus, for each $i$, there exists a sequence $\sigma_j \rightarrow 0$ such  that   $\{v_{\sigma_j,i} = \log_{\sigma_j} \circ u_{\sigma_j,i}\}$ converges locally uniformly to a  map 
 \[
 v_{*,i}: \Omega_i \rightarrow T_{q_0}X.
 \]
 By a diagonalization procedure, we conclude that (after taking a subsequence), 
 $\{v_{\sigma_j}:\D \rightarrow T_{q_0}X\}$ converges locally uniformly in the pullback sense to $v_*:\D \rightarrow T_{q_0}X$.

\begin{definition}
We will call the map  $v_*:\D \rightarrow (T_{u(p_0)}X,\delta)$  an \emph{Alexandrov tangent map} of a harmonic map $u:\D \rightarrow (X,d)$ at $p_0$.  
\end{definition}

\begin{definition}
Let $u:\Sigma \rightarrow (X,d)$ be a harmonic map from a Riemann surface into a locally compact CAT($\kappa$) space, $p_0 \in \Sigma$ and $\D$ a holomorphic disk centered at $p_0$.  A map $u_*$ (resp.~$v_*$) is said to be a \emph{tangent map of $u$ at $p_0$} (resp.~\emph{Alexandrov tangent map of $u$ at $p_0$})  if $u_*$ (resp.~$v_*$) is a tangent map (resp.~Alexandrov tangent map) of $u\big|_\D$.
\end{definition}

\begin{remark}
{ As previously stated, a tangent map $u_*$ is a harmonic map.  This follows from the fact that all blow up maps $u_\sigma$ are harmonic maps (since harmonicity is preserved under the rescaling of the target distance function)  and \cite[Theorem 3.11]{korevaar-schoen2}.  On the other hand,  an  Alexandrov tangent map $v_*$ is not necessarily harmonic.   
In Theorem~\ref{tangent_maps_are_cones}, we show that the local compactness of $X$ and the manifold hypothesis} are sufficient conditions for $v_*$ to be a harmonic map. 
\end{remark}

\subsection{Tangent maps for maps into CAT($\kappa$) manifolds.}

We now specialize to the case when $X$ is a
CAT($\kappa$) manifold (cf.~Definition~\ref{lcs}).

 \begin{theorem}\label{tangent_maps_are_cones}
Let $u:\Sigma \rightarrow (X,d)$ be a harmonic map from a Riemann surface into a CAT($\kappa$) manifold, $p_0 \in \Sigma$ and $\D$ a holomorphic disk centered at $p_0$. Then an Alexandrov tangent map of $u$ at $p_0$ is a tangent map of $u$ at $p_0$.  In particular, $v_*$ is a degree $\alpha=\mathrm{ord}^u(p_0)$ homogeneous harmonic  map.
\end{theorem}

 The proof of Theorem \ref{tangent_maps_are_cones} is a direct consequence of Lemma~\ref{v*} below.

\begin{lemma} \label{v*} 
Let $u: \Sigma \rightarrow (X,d)$ be a harmonic map from a Riemann surface  into a CAT$(\kappa$) space, $p_0 \in \Sigma$ and $\D$ be a holomorphic disk centered at $p_0$.   Furthermore, let $q_0=u(p_0) \in X$ and $\BB:=\BB_r(q_0)$ be a geodesic ball and assume the following:
\begin{enumerate}
\item[(i)] $\partial  \BB$ and $\overline{\BB}$ are compact.
\item[(ii)] For any point $q \in \BB$, there exists a geodesic $\gamma$, containing $q$, from $q_0$ to a point on $\partial \BB$. 
\end{enumerate}
(Note that if  $(X,d)$ is a CAT$(\kappa$) manifold, then (i)  holds since $X$ is locally compact and (ii) holds by \cite[Theorem II.5.12]{bridson-haefliger}.)
\vskip1.5mm
\noindent If  the sequence of  blow up maps $\{u_{\sigma_j}\}$ of $u$ at $p_0$ converges locally uniformly in the pullback sense to a tangent  map $u_*:\
\D \rightarrow X_*$  and $\{v_{\sigma_j}=\log_{\sigma_j} \circ u_{\sigma_j}\}$ converges locally uniformly in the pullback sense to  an Alexandrov tangent map $v_*:\D \rightarrow T_{u(p_0)}X$, then
$\{u_{\sigma_j}\}$
converges  locally uniformly in the pullback sense to $v_*$.  The Alexandrov tangent map $v_*$ is a homogeneous harmonic map with
\begin{equation} \label{u*v*dist}
d_*(u_*(x_0), u_*(x_1))=\delta(v_*(x_0),v_*(x_1)), \ \ \forall x_0, x_1 \in \D.
\end{equation}
   Moreover, the energy density function and the directional energy density functions of $u_{\sigma_j}$ converge weakly to those of $v_*$.
\end{lemma}

\begin{proof}
Rescaling if necessary, we can assume $(X,d)$ is a locally CAT(1) manifold. Fix $i$ and let  $x_0, x_1 \in \Omega_i$.  Throughout this proof, $k=0$ or $k=1$.   For $\sigma_j>0$ sufficiently small, $u(\D_{\sigma_j}) \subset \BB$.   By condition (ii), there exists a geodesic  $\gamma_{k,j}$  from $q_0$ to a point $\hat{p}_{k,j} \in \partial \BB$ containing the point $u_{\sigma_j,i}(x_k)$. Set 
\[
l_{k,j}:=d_{\sigma_j}(u_{\sigma_j,i}(x_k),q_0).
\]
Thus, 
\[
v_{\sigma_j,i}(x_k):=\log_{\sigma_j} \circ u_{\sigma_j,i}(x_k) =(l_{k,j},[\gamma_{k,j}]).
\]
By condition (i),  $\partial \BB$ is compact. Thus, by taking a subsequence if necessary, we can assume 
\begin{equation} \label{catinthehat1}
\mbox{$\hat{p}_{k,j}$ converges to $\hat{p}_k$ as $j \rightarrow \infty$}.
\end{equation}
Let $\gamma_k$ be the geodesic from $q_0$ to $\hat{p}_k$.  For each $j$, consider $\gamma_k$ as a geodesic in $(X,d_{\sigma_j})$  and let $p_{k,j}$  be the point on $\gamma_k$ satisfying
 \[
d_{\sigma_j}(p_{k,j},q_0)=l_{k,j}.
\]  
Since  $(X,d_{\sigma_j})$ is a CAT$(\mu_{\sigma_j}^2)$ space, we use a comparison triangle in the sphere ${\mathbb S}^2_{\mu_{\sigma_j}^2}$ with Gauss curvature $\mu_{\sigma_j}^2$ to define comparison angles.  More specifically,  $\widetilde{\angle}^{(\mu_{\sigma_j})}_{q_0} (p,q)$ is the angle  at $\tilde{q}_0$ of the comparison triangle $\triangle \tilde{q}_0 \tilde{p} \tilde{q}$  in the sphere ${\mathbb S}^2_{\mu_{\sigma_j}^2}$.  By the definition of angles and comparison angles, we have
\begin{equation} \label{2ineq1}
\Theta([\gamma_k],[\gamma_{k,j}]) \leq \widetilde{\angle}^{(\mu_{\sigma_j})}_{q_0} (u_{\sigma_j,i}(x_k),p_{k,j}) \leq \widetilde{\angle}^{{(\mu_{\sigma_j})}}_{q_0}(\hat{p}_{k,j},\hat{p}_{k}).
\end{equation}
From (\ref{catinthehat1}) and  (\ref{2ineq1}), we conclude
\[
\lim_{\sigma_j \rightarrow 0}\delta(v_{\sigma_j,i}(x_k), (l_{k,j},[\gamma_{k}])  )= \lim_{\sigma_j \rightarrow 0} \delta((l_{k,j},[\gamma_{k,j}]),(l_{k,j},[\gamma_{k}])) =0.
\]
Furthermore,  
\[
\lim_{j \rightarrow \infty} l_{k,j} =d_*(u_{*,i}(x_k),u_{*,i}(0))=:l_k,
\]  
and therefore
\[
\lim_{\sigma_j \rightarrow 0}\delta(v_{\sigma_j,i}(x_k), (l_k,[\gamma_{k}])) =\lim_{\sigma_j \rightarrow 0}\delta(v_{\sigma_j,i}(x_k), (l_{k,j},[\gamma_{k}])) =0.
\]
By the definition of $v_{*,i}$, we thus have that 
\[
v_{*,i}(x_k)=(l_k,[\gamma_{k}]).
\]
From (\ref{catinthehat1}) and the second  inequality in (\ref{2ineq1}), we conclude that
\[
\lim_{\sigma_j\rightarrow 0} d_{\sigma_j}(u_{\sigma_j,i}(x_k),p_{k,j}) =0.
\]
Thus,
\begin{equation} \label{moz1}
d_*(u_{*,i}(x_0),u_{*,i}(x_1))=\lim_{\sigma_j\rightarrow 0}  d_{\sigma_j}(u_{\sigma_j,i}(x_0),u_{\sigma_j,i}(x_1)) =\lim_{\sigma_j \rightarrow 0} d_{\sigma_j}(p_{0,j},p_{1,j}).
\end{equation}
Since $p_{0,j} \in \gamma_0$ and $p_{1,j} \in \gamma_{1}$, the definition of angles implies 
\[
\lim_{\sigma_j \rightarrow 0} \widetilde{\angle}^{(\mu_{\sigma_j})}_{q_0} (p_{0,j},p_{1,j}) = \angle (\gamma_0,\gamma_{1}).
\]
Thus,
\begin{equation} \label{art2}
\lim_{\sigma_j \rightarrow 0} d_{\sigma_j}(p_{0,j},p_{1,j}) =\delta((l_0,[\gamma_0]), (l_1,[\gamma_{1}])) = \delta(v_{*,i}(x_0), v_{*,i}(x_{1})). 
\end{equation}
Combining (\ref{moz1}), (\ref{art2}), we obtain
\begin{eqnarray*}
d_*(u_{*,i}(x_0),u_{*,i}(x_1))
& = & 
\delta(v_{*,i}(x_0), v_{*,i}(x_{1})).
\end{eqnarray*}
In particular, this implies that $u_{\sigma_j}$ converges locally uniformly in the pullback sense to $v_*$ and that $v_*$ is homogeneous.

To complete the proof, consider the metric cone $\CC(X)$ over $X$.  More precisely, $(\CC(X),D)$ is an NPC space given by 
\[
\CC(X)= [0,\infty) \times X/\sim,
\]
where $\sim$ identifies all points of the form $(0,p)$ as the vertex ${\mathcal O}$, along with the distance function $D$ given by
\[
D^{{ 2}}((\rho_1, q_1), (\rho_2,q_2)) = \rho_1^2+\rho_2^2-2\rho_1\rho_2 \cos \min \{\pi, d(q_1,q_2)\}.
\]
Furthermore, define the rescaled distance function on $\CC(X)$ by
\[
D_\sigma= \mu_\sigma^{-1}D.
\]
Define the embedding of $X$ into $\CC(X)$ by
\[
\iota: X \hookrightarrow \CC(X), \ \ \iota(q)=(1,q).
\]
The \emph{lift} of the blow up map to $\CC(X)$ is the map   defined by
\[
\bar u_{\sigma_j}:\D \rightarrow (\CC(X),D_{\sigma_j}), \ \ \bar u_{\sigma_j}=\iota \circ u_{\sigma_j}.
\] 
The lift  $\bar u_{\sigma_j}$   has the same energy density and directional energy density functions as those of $u_{\sigma_j}$ and  is  within $\epsilon_j \rightarrow 0$ of minimizing.  Furthermore,  the sequence $\{\bar u_{\sigma_j}\}$ converges locally uniformly in the pullback sense  to $u_*$, and hence to $v_*$ (cf.~\cite[proof of Proposition 7.5]{BFHMSZ}).   
Therefore, by 
 \cite[Theorem 3.11]{korevaar-schoen2}, $v_*$ is harmonic and the
 energy density function and directional energy density functions of $u_{\sigma_j}:\D \rightarrow (X,d_{\sigma_j})$ converge to those of $v_*$.   
\end{proof}

\section{Non-degenerate harmonic maps}
\label{sec:Non-degenerate}
We will now restrict to harmonic maps which satisfy a non-degeneracy condition. Non-degeneracy  generalizes the notion of a full rank differentiable map between manifolds.  
We will show that non-degenerate harmonic maps between surfaces have enough structure to develop a degree theory and exploit some classical results.

This section contains the  proofs of  Theorem~\ref{Main1} (end of Subsection 4.1) and Theorem~\ref{sphere_cor} (end of Subsection 4.2).

\begin{definition}\label{nondegen_def}
We say a harmonic map $u:\Sigma \rightarrow (S,d)$ from a Riemann surface into a locally CAT($\kappa$) manifold  is  \emph{non-degenerate} if any tangent map $u_*:\D \rightarrow (X_*,d_*)$ of $u$ at $p_0 \in \Sigma$ has the property that
\begin{equation}\label{nondeg_eq}
d_*(u_*(z),u_*(0))>0, \ \forall z \in \D \backslash \{0\}.
\end{equation}
\end{definition}

\subsection{Non-degenerate harmonic maps between surfaces are branched covers}\label{locdeg}

\begin{definition} \label{branchset}
The branch set $\BB_u$ of a harmonic map $u:\Sigma \rightarrow (S,d)$ is the set of points $p$ such that $u$ is not a local homeomorphism at $p$.
\end{definition}

We will show that non-degenerate harmonic maps are open and discrete. 
 V\"ais\"al\"a \cite{vaisala} demonstrated the set $\BB_u$ is of topological co-dimension 2 for open and discrete maps between topological manifolds.  Using the order function, we can then  improve this assertion to show that  $\BB_u$ is a discrete set.

 Discreteness  follows immediately from the definition of non-degeneracy and the existence of a tangent map. (Note the discreteness result does not require that the target be a manifold.)
\begin{lemma} \label{discrete}
If $u:\Sigma \rightarrow (X,d)$ is a non-degenerate harmonic map from a Riemann surface into a locally CAT($\kappa$) space, then $u$ is discrete.
\end{lemma}

\begin{proof}
On the contrary, assume that $u$ is not discrete; i.e.~there exist $q_0 \in X$ and a sequence $p_j \rightarrow p_0$ such that $u(p_j)=q_0$.  Let $\D$ be a holomorphic disk centered at $p_0$ and let $z_j \in \D$ correspond to $p_j$.  Let $\sigma_j=2|z_j|$  and consider a  sequence of blow up  maps $\{u_{\sigma_j}\}$ of $u\big|_\D$.  
By taking a subsequence if necessary, we can assume that the sequence  $\{u_{\sigma_j}\}$ converges locally uniformly in the pullback sense to a  tangent map $u_*$ and that the sequence $
\{\zeta_j= \frac{z_j}{\sigma_j}  \} \subset  \partial \D_{\frac{1}{2}}(0)$ converges to $\zeta_0 \in \partial \D_{\frac{1}{2}}(0)$.  
Since $u_{\sigma_j}(\zeta_j)=u_{\sigma_j}(0)$,  we have $u_*(\zeta_0)=u_*(0)$.  This contradicts the fact that $u$ is a non-degenerate map.
\end{proof}

To prove openness of $u$, we exploit the structure of the Alexandrov tangent cone $T_{q}S$ for a point $q$ in a  CAT($\kappa$) surface $(S,d)$.  Indeed,  
Proposition~\ref{coneovercurve} asserts that $T_{q}S$ is a cone over a finite length closed curve. 
Thus there exists an orientation preserving (with respect to the orientation of $T_qS$ inherited  from $S$)  isometry
\begin{equation} \label{Iq}
I_q:  T_{q}S \rightarrow ({\mathbb C}, ds^2)
\end{equation}
where 
\begin{equation} \label{localmetric}
ds^2 = \beta^2|z|^{2(\beta-1)}|dz|^2
\end{equation}
for a suitable constant $\beta \geq 1$.  The constant $\beta$  is determined by the cone angle of $T_qS$; indeed, the curvature measure of $ ({\mathbb C}, ds^2)$ is $2\pi(1-\beta) \delta_0$ where $\delta_0$ is a Dirac measure at the origin.  

Kuwert \cite[Lemma 3]{kuwert} classified all homogeneous, harmonic maps from $\C$ to $(\C, ds^2)$.   Accordingly, we have the following: 
\begin{itemize}
\item 
If  an Alexandrov tangent map $v_*$ satisfies (\ref{nondeg_eq}), then up to orientation and rotation (and with $\alpha=\mathrm{ord}^u(p)$)

\begin{equation}\label{tangent_map_equation}
I_{u(p)} \circ v_*(z)= \left\{\begin{array}{ll}
c z^{\alpha/\beta}  & \text{if } k=0,\\
c \left(\frac 12\left(k^{-\frac 12}z^\alpha + k^{\frac 12} \bar z^\alpha\right)\right)^{1/\beta} & \text{if } 0<k<1.
\end{array}\right.
\end{equation}
with
\begin{equation} \label{alphabeta}
\alpha/\beta \in \N.
 \end{equation}

\item If $v_*$ does not satisfy (\ref{nondeg_eq}), then there exist a finite number of disjoint sectors of $\D$ such that $v_*$ maps each sector to a geodesic ray.  In this case, $k=1$.
\end{itemize}
Here $k=k_u(p) \in [0,1]$ is the \emph{stretch} of $u$ at $p$ and the constant $c$ in (\ref{tangent_map_equation}) is determined by the normalization (cf.~(\ref{bmnorm}) and $W^{1,2}$-trace theory \cite[Theorem 1.12.2]{korevaar-schoen1})
\begin{equation}\label{normalizing_eq}
\int_{\partial \D} \delta^2(v_*,\OO) \, d\theta =1.
\end{equation}
\begin{remark}\label{stretch_defn}
If  any tangent map of $u$ at $p$ satisfies (\ref{nondeg_eq}), then all tangent maps of $u$ at $p$ satisfy (\ref{nondeg_eq}). To see this, first note that Lemma~\ref{v*} implies that a tangent map $u_*$ satisfies \eqref{nondeg_eq} if and only if its corresponding Alexandrov tangent map $v_*$ does. From the characterization given above, $v_*$ satisfies (\ref{nondeg_eq}) if and only if $k_u(p)\neq 1$. Finally, the value of $k$ is independent of the choice of  tangent map (cf. \cite[Lemma 5]{kuwert}); that is the \emph{stretch function} is a well defined function $k_u:\Sigma \rightarrow [0,1]$. Therefore, as soon as one tangent map $u_*$ satisfies \eqref{nondeg_eq}, $k_u(p)\neq 1$ and thus all tangent maps satisfy \eqref{nondeg_eq}.   
\end{remark}

\begin{proposition} \label{omt}
A non-degenerate harmonic map $u:\Sigma \rightarrow (S,d)$ from a Riemann surface into a CAT($\kappa$) surface is an open map.
\end{proposition}

\begin{proof}Let $\UU \subset \Sigma$ be an open set, $p_0 \in \UU$, $q_0 =u(p_0)$ and $\D \subset \UU$ be a holomorphic disk centered at $p_0$.
  Let $\{u_{\sigma_j}\}$ be a sequence of blow up maps of $u\big|_\D$   converging locally uniformly in the pullback sense to  $v_*:\D \rightarrow T_{q_0}S$.  By (\ref{tangent_map_equation}), $v_*(\D_\frac{1}{2})$ is an open set and hence    $ \BB^{T_{q_0}S}_\rho(\mathcal O) \subset v_*(\D_{\frac{1}{2}})$ for some $\rho>0$.  Thus, for sufficiently small $\sigma_j$, the geodesic disk $\BB_{\rho}^{(S,d_{\sigma_j})}(q_0)$ is contained in $u_{\sigma_j}(\D)$.  
Equivalently, $\BB_{\sigma_j \rho}^{(S,d)}(q_0)\subset u(\D_{\sigma_j})\subset u(\mathcal U)$.
\end{proof}

The proof that $u$ is a branched cover requires that points with high order are discrete. The proof of this discreteness does not require the Alexandrov tangent map be harmonic, and is thus true for a larger class of maps. In the following lemma, we do not require that $u$ be non-degenerate or $(X,d)$ be a manifold.
\begin{lemma}  \label{A}
Let $u:\Sigma \rightarrow (X,d)$ be a harmonic map from a Riemann surface into a locally CAT($\kappa$) space. Then  
$\Aa:=\{z \in \Sigma:  \alpha_u(z)= \mathrm{ord}^u(z) \geq  2\}$
is a discrete set.
\end{lemma}

\begin{proof}
Suppose there exists $\{p_j\} \subset \Aa$ such that  $p_j \rightarrow p_0$. Let $\D$ be a holomorphic disk centered at $p_0$ and let $z_j \in \D$ correspond to $p_j$.  Let $\sigma_j=2|z_j|$  and consider a  sequence of blow up  maps $\{u_{\sigma_j}\}$ of $u$ at $0$.    
By taking a subsequence if necessary, we can assume that the sequence  $\{u_{\sigma_j}\}$ converges locally uniformly in the pullback sense to a  tangent map $u_*$ and that the sequence $
\{\zeta_j= \frac{z_j}{\sigma_j}  \} \subset  \partial \D_{\frac{1}{2}}(0)$ converges to $\zeta_0 \in \partial \D_{\frac{1}{2}}(0)$.  By Lemmas \ref{order_equal} and \ref{order_upper}, $\limsup_{j \rightarrow \infty} \alpha_{u_{\sigma_j}}(\zeta_j)  \leq \alpha_{u_*}(\zeta_0) =1$.   Thus, we have that $\alpha_{u_{\sigma_j}}(\zeta_j)<2$ for $j$ sufficiently large which in turn implies $\alpha_u(z_j)< 2$, a contradiction.
\end{proof}

\begin{theorem} \label{itsbranched}
If $u:\Sigma \rightarrow (S,d)$ is  a proper, non-degenerate harmonic map from a Riemann surface to an oriented locally CAT($\kappa$) surface, then  $u$ is a branched cover.
\end{theorem}

\begin{proof}
Let  $\deg_K f$ be as defined in \cite[Definition VIII.4.2]{dold}.   Since $u$ is a discrete map, for any $p_0 \in \Sigma$ with $q_0=u(p_0)$, there exists a connected, simply connected neighborhood $U$ of $p_0$ such that    
\begin{equation} \label{onlyonept}
\{p_0\}=U \cap u^{-1}(q_0).
\end{equation}
Thus, $\deg_{q_0} u|_{V} = \deg_{q_0} u|_{U}$  for any neighborhood $V \subset U$ of $p_0$.
  Since $u$ is an open map,
\[
\Z =H_1(\Sp^1)=H_2(\bar \D, \partial \D) = H_2(U,U \backslash p_0) = H_2(u(U),u(U) \backslash q_0)
\]
and 
\[
\deg_{q_0} u|_{U} = (u|_{U})_{\#} (1)
\]
where 
\[
(u|_{U})_{\#}: H_2(U,U \backslash p_0) \rightarrow  H_2(u(U),u(U)\backslash q_0)
\]
 is the induced homomorphism of the local homology  groups. Thus,  
$\deg_{q_0} u|_{U}$
  is the (signed) winding number of the curve $u\circ \gamma$ around $q_0$ where $\gamma$ is a positively oriented parameterization of  $\partial \D$ where $\D$ is a conformal disk centered at $p_0$ compactly contained in $U$.  We denote
  \[
  w_{\#}(p_0):=\deg_{q_0} u|_{U} = (u|_{U})_{\#} (1).
  \]
The integer $w_{\#}(p_0)$ satisfies the following properties:

\begin{itemize}
\item[(i)]  \emph{$|w_{\#}(p_0) |$ is equal to $\alpha/\beta$ in (\ref{alphabeta}) for every tangent map of $u$ at $p_0$.}   Indeed, for $\sigma>0$ sufficiently small, the map $\log_{q_0}: u_\sigma(\D) \rightarrow \log_{q_0}(u_\sigma(\D)) \subset T_{q_0}S$ is a homotopy equivalence with $\log_{q_0}^{-1}(\OO)=\{q_0\}$ (cf.~Subsection~\ref{alex}).  Thus, the claim follows from  the uniform convergence of $v_{\sigma_i}=\log_{q_0}\circ u_{\sigma_i}$ to $v_*$  as asserted in  Lemma~\ref{v*}.
\item[(ii)]  \emph{Either $w_{\#}(p)=1$ for all $p \in \Sigma \backslash \BB_u$ or   $w_{\#}(p)=-1$ for all $p \in \Sigma \backslash \BB_u$.}  To see this, first note that 
  $|w_{\#}(p)|=1$ for all $p \in \Sigma \backslash \BB_u$  since $u$ is a local homeomorphism on $\Sigma \backslash \BB_u$.  But since  $\Sigma \backslash \BB_u$ is connected (because $\dim \BB_u=0$ by \cite[Theorem 5.4]{vaisala}), we conclude that  either $w_{\#}(p)=1$ for all $p \in \Sigma \backslash \BB_u$ or   $w_{\#}(p)=-1$ for all $p \in \Sigma \backslash \BB_u$.  
  \item[(iii)] \emph{Either $w_{\#}(p) >0$ for all $p_0 \in \Sigma$ or $w_{\#}(p)<0$ for all $p_0 \in \Sigma$.}  To see this, assume without the loss of generality that $w_{\#}(p)=1$ for all $p \in \Sigma \backslash \BB_u$.   For $p_0 \in \BB_u$, we  need to show $w_{\#}(p_0) >0$.  Let $U$ be as in (\ref{onlyonept}).   With $q_0=u(p_0)$, let 
  \begin{equation} \label{choiceB}
  B:=\BB_\epsilon(q_0) \mbox{ be such that } \bar B  \subset u(U) \mbox{ and } V:=u^{-1}(B).
  \end{equation}      
By \cite[Proposition VIII.4.4]{dold}, 
$\deg_{\bar B} u|_V=\deg_q u|_V$ for all $q \in \bar B$.
In particular, 
 \[
w_{\#}(p_0)= \deg_{q_0} u|_V  =\deg_{q} u|_V, \ \ \forall q \in B.
 \]
 Since $u^{-1}(B)$ is an open set and $\dim(\BB_u)=0$ (cf.~\cite[Theorem 5.4]{vaisala}), $B \backslash u(\BB_u) \neq \emptyset$.  
For $q \in B \backslash u(\BB_u)$, let  $(u|_U)^{-1}(q) :=\{p_1,\dots, p_k\}$ and $\{V_i\}_{i=1}^k$ be an open cover of $V$ such that  each $V_i$ contains exactly one element $p_i$ of $u^{-1}(q)$. Thus,  \cite[Proposition VIII.4.7]{dold} implies 
\begin{equation} \label{addition}
w_{\#}(p_0)  = \deg_{q} u|_V=\sum_{i=1}^k  \deg_{q}u|_{V_i} = \sum_{i=1}^k w_{\#}(p_i).
\end{equation}
Since $q \in B \backslash u(\BB_u)$, we have that $w_{\#}(p_i)=1$ and thus $w_{\#}(p_0) =k>0$.
 \end{itemize} 
 
 To show that $u$ is a branched cover, we need to show that $\BB_u$ is a discrete set and $u$ is an even covering away from $\BB_u$.  By \cite[Proposition and Definition VIII.4.5]{dold}, $u$ is an even covering on $\Sigma \backslash \BB_u$.

 Since   $\Aa=\{z \in \Sigma:  \mathrm{ord}^u(z) \geq 2\}$ is a discrete set by Lemma~\ref{A}, to prove $\BB_u$ is discrete it is sufficient to show that
\begin{equation}\label{bu_string}
\BB_u \subset \DD \subset \Aa
\end{equation}
where
\[
\DD=\{p \in \Sigma: \left| w_{\#}(p) \right| \neq 1\}.
\]  

To show the  inclusion on the right in \eqref{bu_string}, let $p_0 \in \DD$.  This implies that $\alpha/\beta \geq 2$ in (\ref{alphabeta}) for every tangent map of $u$ at $p_0$. Hence, $\alpha=\mathrm{ord}^{v_*}(0)=\mathrm{ord}^u(p_0) \geq 2$, and $p_0 \in \Aa$.

To show the inclusion on the left in \eqref{bu_string}, we show $p_0 \notin \DD \Rightarrow p_0 \notin \BB_u$; in other words, $u$ is a local homeomorphism at $p_0 \in \Sigma \backslash \DD$.
Assume without the loss of generality that $\deg_{q_0} u|_U =1$ instead of $-1$, where $U$ is as in (\ref{onlyonept}). Following (\ref{addition}), 
\[
1=\deg_{q_0} u|_V =\sum_{i=1}^k w_{\#}(p_i) \mbox{ where } (u|_V)^{-1}(q):=\{p_1, \dots, p_k\} 
\]
for any $q \in B=\BB_\epsilon(q_0)$ where $B$ and $V=u^{-1}(B)$ are as in (\ref{choiceB}).
 Since, by (iii), $w_{\#}(p)$ has the same sign for all $p \in V$, the set $(u|_V)^{-1}(q)$ must consist of exactly one element for each $q \in B$; i.e.~$u|_V:V \rightarrow B$ is injective.
Thus  $u|_V: V \rightarrow B$ is an open, continuous bijection, and hence  a homeomorphism. 
   \end{proof}

 \begin{corollary}\label{HomeoCor}
Suppose $u:\Sigma \rightarrow (S,d)$ is a proper, non-degenerate harmonic map from a Riemann surface to an oriented locally CAT($\kappa$) surface. If $u$ is degree 1, then $u$ is a homeomorphism.
 \end{corollary}
 
\begin{proof}
By \cite[Proposition VIII.4.5 and Proposition VIII.4.7]{dold}, 
 \[
1= \deg(u) = \sum_{\{p \in \Sigma: u(p)=q\}} w_{\#}(p).
\]
Since $w_{\#}(p)$ has the same sign for all $p \in \Sigma$, $\{p: u(p)=q\}$ must consist of exactly one element for each $q \in S$. Thus  $u$ is an open, continuous bijection, and hence  a homeomorphism. 
\end{proof}

\begin{proofMain1} 
Theorem~\ref{itsbranched} asserts that such a map is a branched cover.  If this map is degree 1, then Corollary~\ref{HomeoCor} implies that it is a homeomorphism. 
\end{proofMain1}

\subsection{Non-degenerate maps and $H(k)$-quasiconformality} \label{Hksection}

Recall  the geometric notion of quasiconformality (cf.~\cite{heinonen}).

\begin{definition}\label{qc_def}
For a homeomorphism $u:(X,d_X) \rightarrow (Y,d_Y)$ between metric spaces, define $H_u:X \rightarrow [1,\infty)$ by setting
\[
H_u(p) :=\limsup_{r \rightarrow 0} \frac{L_u(p,r)}{l_u(p,r)} 
\]
where 
\begin{eqnarray*}
L_u(p,r) & = & \max_{d_X(p,q) = r} d_Y(u(p),u(q)),\\
l_u(p,r) & = & \min_{d_X(p,q) = r} d_Y(u(p),u(q)).
\end{eqnarray*}

A map  $u:(X,d_X) \rightarrow (Y,d_Y)$ between metric spaces  is said to be \emph{$H$-quasiconformal} if $u$ is a homeomorphism and $H_u(p) \leq H$ for all $p \in X$.
\end{definition}
 We let $\Sigma$ denote a Riemann surface and fix a conformal metric $g$ on $\Sigma$ of constant curvature $-1, 0,$ or $1$. 
 Let $u:\Sigma \rightarrow (S,d)$ be a non-degenerate harmonic homeomorphism from $\Sigma$ to a locally CAT($\kappa$) surface. In constructing the tangent map we choose normal coordinates with respect to this metric $g$. The tangent map structure of \eqref{tangent_map_equation} and the definition of $H$ imply that for each $p \in \Sigma$, if $\alpha=\mathrm{ord}^u(p)$ and $k=k_u(p)$, then
  \begin{equation} \label{Hup}
H_{v_*}(0):= \left\{\begin{array}{ll}1, & \text{if } k=0,\\
 \left( \frac{k^{-\frac{1}{2}} + k^{\frac{1}{2}}}{k^{-\frac{1}{2}} - k^{\frac{1}{2}}} \right)^{\frac{1}{\alpha}},& \text{if } k\in (0,1).
 \end{array}\right.
\end{equation}
Note that $\alpha$ and $k$ are independent of the choice of tangent map and thus this structure holds for \emph{all} (Alexandrov) tangent maps of $u$ at $p$.
\begin{lemma}\label{Hklemma}
Let $u:\Sigma \rightarrow (S,d)$ be a non-degenerate harmonic homeomorphism from a Riemann surface to a locally CAT($\kappa$) surface. 
Then, for every $p_0 \in \Sigma$,
\[
H_u(p_0) = H_{v_*}(0)
= H_{u^{-1}}(u(p_0))
\]
where $v_*$ is a tangent map of $u$ at $p_0$.
\end{lemma}
\begin{proof}
Let $p_0 \in \Sigma$.  Use normal coordinates with respect to the Riemannian metric $g$ to identify a neighborhood of $p_0$ with a disk $\D$ and $p_0$ with the origin $0 \in \D$.  Let $r_i \rightarrow 0$ be such that
\[
H_u(0) 
=\lim_{i \rightarrow \infty}  \frac{ L_u(0,r_i)}{ l_u(0,r_i)}.
\]
Let $ z_i'$, $ \zeta_i' \in \D$ be points such  that
\[
| z_i'|=r_i=| \zeta_i'|,
\ \ 
L_u(0,r_i)= d(u(0),u( z_i')) \ \mbox{ and } \  l_u(0,r_i)= d(u(0),u( \zeta_i')).
\]
Let $\sigma_i=2r_i$,  $ z_i= \frac{ z_i'}{\sigma_i}$ and $ \zeta_i = \frac{ \zeta_i'}{\sigma_i}$.  Note that $|{z}_i|=|{\zeta}_i|=\frac{r_i}{\sigma_i}= \frac{1}{2}$.  Taking a subsequence if necessary, we can assume $ z_i \rightarrow  z_{\infty}$, $ \zeta_i \rightarrow  \zeta_{\infty}$ and $u_{\sigma_i}$ converges locally uniformly to $v_*$.  Then 
\[
H_u(0) =\lim_{i \rightarrow \infty} \frac{L_u(0,r_i)}{l_u(0,r_i)} =\frac{\delta(v_*(0),v_*({z}_{\infty}))}
{\delta(v_*(0),v_*({\zeta}_{\infty}))} \leq H_{v_*}(0).
\]
Next, note that by homogeneity
\[
H_{v_*}(0) =  \frac{L_{v_*}(0,\frac{1}{2})}{l_{v_*}(0,\frac{1}{2})}.
\]
Now let  $\hat z_\infty, \hat \zeta_\infty \in \D$ with $|\hat z_\infty|= |\hat \zeta_\infty|=\frac{1}{2}$ such that
\[
H_{v_*}(0) =\frac{\delta(v_*(0),v_*(\hat{z}_{\infty}))}
{\delta(v_*(0),v_*(\hat{\zeta}_{\infty}))} =
\lim_{\sigma_i \rightarrow 0}  \frac{d_{\sigma_i} (u(0),u( \hat z_\infty))}
{d_{\sigma_i} (u(0),u( \hat\zeta_\infty))} 
= \lim_{\sigma_i \rightarrow 0} \frac{d(u(0), u(\sigma_i  \hat z_\infty))}{d(u(0), u(\sigma_i \hat \zeta_\infty))}\leq H_u(0).
\]
Thus,
 $H_u(0)=H_{v_*}(0)$.

We will next show that for $P_0:= u(p_0)$, $H_{u^{-1}}(P_0) = H_{v_*}(0)$.  
 Let $\rho_i \rightarrow 0$, $P_i,Q_i \in S$  such that 
\[
H_{u^{-1}}(P_0) = \limsup_{\rho \rightarrow 0} \frac{L_{u^{-1}}(P_0,\rho)}{l_{u^{-1}}(P_0,\rho)}=\lim_{i \rightarrow \infty}  \frac{L_{u^{-1}}(P_0,\rho_i)}{l_{u^{-1}}(P_0,\rho_i)},
\]
\[
d(P_0,P_i) = \rho_i \mbox{  and } \ 
L_{u^{-1}}(P_0,\rho_i)= d_g(u^{-1}(P_0),u^{-1}(P_i)),
\]
\[
d(P_0,Q_i) = \rho_i \mbox{  and } \ 
l_{u^{-1}}(P_0,\rho_i)= d_g(u^{-1}(P_0),u^{-1}(Q_i)).
\]
Use normal coordinates centered at $u^{-1}(P_0)$ with respect to the Riemannian metric $g$ to identify a neighborhood of $u^{-1}(P_0)$ with a disk $\D$  and $u^{-1}(P_0)$ with the origin $0 \in \D$.  Assuming $\rho_i$ is sufficiently small, let $ z_i',  \zeta_i' \in \D$ be such that
\[
 z_i'=u^{-1}(P_i) \mbox{ and }  \zeta_i'=u^{-1}(Q_i)
\]  
which implies
\[
H_{u^{-1}}(P_0)=\lim_{i \rightarrow \infty} \frac{| z_i'|}{| \zeta_i'|}.
\]
Let  $\sigma_i =2| z_i'|$,  $ z_i=\frac{ z_i'}{\sigma_i}$, $ \zeta_i=\frac{ \zeta_i'}{\sigma_i}$.  Thus, $| z_i|= \frac{1}{2}$ and $| \zeta_i| \leq \frac{1}{2}$.       There exists a subsequence of blow up maps $u_{\sigma_i}$, which we denote again by $u_{\sigma_i}$,  that converge locally uniformly (after rotation if necessary) to a tangent map $v_*$ which again satisfies \eqref{tangent_map_equation} and \eqref{Hup}.

We can assume $ z_i \rightarrow  z_\infty$, $ \zeta_i \rightarrow  \zeta_\infty$ and
\[
\arg  z_\infty=\frac \pi{2\alpha} \ \mbox{ and } \ \arg  \zeta_\infty = 0.
\]
Moreover, $| z_\infty|=\frac{1}{2}$ and for $\delta$ induced by the metric $\beta^2 |z|^{2(\beta-1)} |dz|^2$,
\begin{eqnarray*}
\lim_{i \to \infty} d_{\sigma_i}(P_0, P_i) & = &  \delta(0,v_*( z_\infty))
\\
& = &  \int_0^{|v_*( z_\infty)|} \beta t^{\beta-1} dt
\\
& = &  |v_*( z_{\infty})|^\beta
\\
& = & \left\{ \begin{array}{ll} \frac {c^\beta}{2^\alpha}, & \text{ if } k=0,\\
\frac{ c^\beta}2  \left( k^{-\frac{1}{2}} \left( \frac{1}{2} \right)^\alpha-k^{\frac{1}{2}} \left( \frac{1}{2} \right)^\alpha \right) ,  & \text{ if } k \in (0,1)
\end{array}\right.
\\
& = & \left\{\begin{array}{ll} \frac {c^\beta} {2^\alpha},& \text{ if } k=0,\\
\frac {c^\beta}{2^{\alpha+1}}  \left(k^{-\frac{1}{2} }- k^{\frac{1}{2} }\right),& \text{ if } k \in(0,1).
\end{array}\right.
\end{eqnarray*}
Similarly, with  $| \zeta_\infty|=:r\leq \frac 12$, we obtain
\[
\lim_{i \to \infty} d_{\sigma_i}(P_0, P_i)= \lim_{i \to \infty} d_{\sigma_i}(P_0, Q_i)=\left\{\begin{array}{ll}\frac {c^\beta}{r^\alpha},& \text{ if } k=0\\ \frac{c^\beta}2 r^{\alpha}   \left( k^{-\frac{1}{2}} + k^{\frac{1}{2}}  \right), &\text{ if }k\in(0,1).\end{array}\right.
\]
Combining the above two equalities and solving for $r$, we obtain
\[
H_{u^{-1}}(P_0)=\lim_{i \rightarrow \infty} \frac{| z_i|}{| \zeta_i|} = \frac{| z_\infty|}{| \zeta_\infty|} =\frac{1}{2r}=\left\{\begin{array}{ll}1, &\text{ if }k=0\\ \left( \frac{k^{-\frac{1}{2}} + k^{\frac{1}{2}}}{k^{-\frac{1}{2}} - k^{\frac{1}{2}}} \right)^{\frac{1}{\alpha}},& \text{ if }k \in(0,1)\end{array}\right.=H_{v_*}(0).
\]
\end{proof}

\begin{proposition}\label{Main2}
Let $u:\Sigma \rightarrow (S,d)$ be a non-degenerate harmonic homeomorphism from a Riemann surface (with the fixed constant curvature metric $g$ chosen as before) to a locally CAT($\kappa$) surface and let $k_u(p)$ denote the {stretch of $u$ at $p$} (cf. Remark \ref{stretch_defn}). Then $u$ and $u^{-1}$ are both $H(k)$-quasiconformal in the metric space sense with 
  \begin{equation} \label{Hustar}
H({k}):= \left\{\begin{array}{ll}1, & \text{if } k=0,\\
\frac{k^{-\frac{1}{2}} + k^{\frac{1}{2}}}{k^{-\frac{1}{2}} - k^{\frac{1}{2}}} ,& \text{if } k \in (0,1)
 \end{array}\right.
\end{equation}if and only if $k := \sup\{k_u(p): p \in \Sigma\} \in [0,1)$. 
 \end{proposition}
 
 \begin{proof}
By Lemma \ref{Hklemma} and \eqref{Hup}, for each $p \in \Sigma$, if $k_u(p) \in [0,1)$ then with $\alpha=\mathrm{ord}^u(p)$ 
\[
H_u(p)=H_{u^{-1}}(u(p)) = \left\{\begin{array}{ll}1, & \text{if } k_u(p)=0,\\
 \left( \frac{k_u(p)^{-\frac{1}{2}} + k_u(p)^{\frac{1}{2}}}{k_u(p)^{-\frac{1}{2}} - k_u(p)^{\frac{1}{2}}} \right)^{\frac{1}{\alpha}},& \text{if } k_u(p)\in (0,1).\end{array}\right.
\]Let $k:= \sup\{k_u(p): p \in \Sigma\}$. If $k \in [0,1)$ then since $\alpha=\mathrm{ord}^u(p) \geq 1$ for all $p\in \Sigma$, we immediately conclude the result with $H_k$ as in \eqref{Hustar}. 

On the other hand, if $k=1$ then there exist $p_i \in \Sigma$ such that $H_u(p_i) \to \infty$ and thus $u$ cannot be $H$-quasiconformal for any $H$.
\end{proof}

\begin{lemma}\label{nd_lemma}
A non-trivial almost conformal harmonic map $u:\Sigma \rightarrow (S,d)$ from a Riemann surface to a locally CAT($\kappa$) surface is non-degenerate. Indeed, an Alexandrov tangent map $v_*$ of $u$ at $p_0 \in \Sigma$ is of the form $I_{u(p_0)} \circ v_*=\frac{z^{\alpha/\beta}}{\sqrt{2\pi}}$  with $\frac{\alpha}{\beta} \in \N$ where $I_{u(p_0)}$ is as in (\ref{Iq}),  $\alpha=\text{ord}^u(p_0)$ and $\beta$ is as in  (\ref{localmetric}).
\end{lemma}

\begin{proof}
 Let $p_0 \in \Sigma$ and $u_{\sigma_j}$ be the sequence of  blow up maps that  converges locally uniformly to $v_*:  \D \rightarrow T_{u(p_0)}S$ (cf.~Lemma~\ref{v*}).     Then the Hopf differentials $\Phi_{u_{\sigma_j}}$ converge to $\Phi_{v_*}$.  Since $u$ is almost conformal, so is $u_{\sigma_j}$ and thus $\Phi_{u_{\sigma_j}} \equiv 0$.  Since the directional energy densities of $u_{\sigma_j}$ converge weakly to those of $v_*$, $\Phi_{v_*} \equiv 0$, and hence $v_*$ is an almost conformal map; i.e.~$I_{u(p_0)} \circ v_*=cz^{\alpha/\beta}$ (cf.~(\ref{tangent_map_equation})).  By \eqref{normalizing_eq}, $c=\frac{1}{\sqrt{2\pi}}$.
\end{proof}

\begin{proofSphereCor}
Let $u:\Sp^2 \rightarrow (S,d)$ be a non-trivial harmonic map from the standard sphere to an oriented locally CAT$(\kappa$) sphere. Then $u$ is almost conformal by  Lemma \ref{holomorphic0} which in turn implies $u$ is non-degenerate by Lemma~\ref{nd_lemma}. It follows from Theorem \ref{Main1} that $u$ is a branched cover, and if the degree of $u$ is 1 then $u$ is a homeomorphism. The structure of $v_*$ at every point of $\Sp^2$ given by Lemma~\ref{nd_lemma} implies that $k=0$ for all $p \in \Sp^2$. Thus, the 1-quasiconformal assertion follows from Proposition \ref{Main2}. 
\end{proofSphereCor}

\section{Proof of Theorem~\ref{Main3}} \label{pachelbel}

In this section, we prove Theorem \ref{Main3}. First, in Subsection \ref{unique_section} we use the structure of an almost conformal Alexandrov tangent map $v_*$ given in Lemma \ref{nd_lemma} and the weak differential inequalities satisfied by the conformal factor $\lambda_u$ given in Theorem \ref{lambdadifeq} to prove the uniqueness statement in Theorem~\ref{Main3}. In Subsection \ref{coarea_section}, we use the approximate metric differential and its structure for finite energy maps given in Lemma \ref{KSjac} coupled with the coarea formula to relate the Hausdorff measure of the image of a map to its total energy. Finally, in Subsection \ref{final_section} we prove the main theorem.

 \subsection{Uniqueness of almost conformal harmonic homeomorphisms}\label{unique_section}
The goal of this subsection is to prove a uniqueness statement for an almost conformal harmonic homeomorphism (cf.~Proposition \ref{unique_mobius}). We start with some preliminary results that  rely heavily on Theorem \ref{lambdadifeq} (cf. \cite{mese}) and the representation of Alexandrov tangent maps given in Lemma \ref{nd_lemma}.

\begin{lemma}  \label{cf}
If $u:\Sigma \rightarrow (S,d)$ is an almost conformal harmonic map from a Riemann surface into a locally CAT($\kappa$) space and $\D$ is a holomorphic disk  with conformal coordinates $z=x+iy$  centered at  $p \in \Sigma$, then there exists an $L^1$-representative $\tilde \lambda_u$ of the conformal factor $\lambda_u$  such that 
\[
 \lim_{\sigma \rightarrow 0} \frac{1}{2\pi \sigma} \int_{\partial \D_\sigma(z_0)} \lambda_u \, d\HH^1= \lim_{\sigma \rightarrow 0} \frac{1}{\pi \sigma^2} \int_{\D_\sigma(z_0)} \lambda_u \, dxdy  = \tilde \lambda_u(z_0), \ \ \forall z_0 \in \D
 \]
 where $\D_\sigma(z_0)=\{z \in \D:  |z-z_0|<\sigma\}$. 
 \end{lemma}

 \begin{proof}
 $\lambda_u$ satisfies the  weak differential inequality $\triangle \lambda \geq -2\kappa \lambda^2$  and is locally bounded (cf.~Theorem~\ref{lambdadifeq}).  Thus,  for any $z_0 \in \D$, 
 \[
 \lim_{\sigma \rightarrow 0} \frac{1}{2\pi \sigma} \int_{\partial \D_\sigma(z_0)} \lambda_u \, d\HH^1 = \lim_{\sigma \rightarrow 0}  
\frac{1}{\pi \sigma^2} \int_{\D_\sigma(z_0)} \lambda_u \, dxdy \mbox{ exists}
\]
by the mean value inequality.  By the  Lebesgue differentiation  theorem, the function $\tilde \lambda_u$ given by this limit at every point of $\D$ is in the $L^1$-class of the conformal factor.
 \end{proof}
 
 \begin{remark}
 We will henceforth denote the $L^1$-representative $\tilde \lambda_u$ in Lemma~\ref{cf} by $\lambda_u$. 
  \end{remark}
 
\begin{lemma} \label{zeroes}For $u$ as in Lemma \ref{cf}, 
if $\lambda_u$ is the conformal factor of $u$ in $\D$ and $\ZZ_u=\{z \in \D:  \lambda_u(z)=0\}$, then $\dim_\HH(\ZZ_u)=0$. 
\end{lemma}

\begin{proof}
The conformal factor $\lambda_u$ satisfies the weak differential inequality $\triangle \log \lambda_u \geq -2 \kappa \lambda_u$ and is locally bounded  (cf.~Theorem~\ref{lambdadifeq}).  Thus, the result follows from the standard theory of subharmonic functions (cf.~\cite{hayman-kennedy}).
\end{proof}

\begin{lemma}  \label{An}
If $u:\Sigma \rightarrow (S,d)$ is an almost conformal harmonic map from a Riemann surface into a locally CAT($\kappa$) surface and $\D$ is a holomorphic disk in $\Sigma$, then $A :=\{z \in \D:  \mathrm{ord}^u(z) >1\}$ is a countable set.\end{lemma}

\begin{proof}
 Since $A=\bigcup_{n \in \N} A_n$ where $A_n=\{z \in \D:  \mathrm{ord}^u(z) >1+\frac{1}{n}\}$, it is sufficient to show that $A_n$ is a discrete set. But this follows immediately from the proof of Lemma \ref{A}. 
\end{proof}

\begin{lemma}  \label{dl}
If $u:\Sigma \rightarrow (S,d)$ is an almost conformal harmonic map from a Riemann surface into a locally CAT($\kappa$) surface and $\D$ is holomorphic disk in $\Sigma$, then 
\[
\lim_{z \rightarrow z_0} \frac{d^2(u(z),u(z_0))}{|z-z_0|^2} =  \lambda_u(z_0), \ \ \forall z_0 \in \D
\] 
where $\lambda_u$ is the conformal factor of $u$ in $\D$. 
\end{lemma}

\begin{proof} 
Let $z_0 \in \D$ and without loss of generality, assume $z_0=0$.  
It is sufficient to show that every sequence $\frac{d(u(z_i),u(0))}{|z_i|}$ with $z_i\rightarrow 0$ has a subsequence that converges to $\lambda_u(0)$.  
Let $\sigma_i=2|z_i|$.  Choose a subsequence $\sigma_{i'} \rightarrow 0$ such that   $\hat z_{i'}=\frac{z_{i'}}{\sigma_{i'}} \rightarrow \hat z_* \in \partial \D_{\frac{1}{2}}$ and  the sequence $\{u_{\sigma_{i'}}\}$ of blow up maps converges locally uniformly in the pullback sense to an Alexandrov  tangent map $v_*$.  By Lemma \ref{nd_lemma}, $v_*$ is identified with the complex-valued function $I_{u(0)} \circ v_*:\C\rightarrow (\C,ds^2)$ given by $I_{u(0)} \circ v_*(z)=\frac{z^{\alpha/\beta}}{\sqrt{2\pi}}$ where $\alpha=\text{ord}^u(0)$ and $ds^2 = \beta^2|z|^{2(\beta-1)}|dz|^2$.  Thus, 
\[
\delta^2(v_*(\hat z_*),v_*(0))  =\frac{1}{2\pi \cdot 2^{2\alpha}}.
\]
By \eqref{mu}, \eqref{orderlimitexists}, and Lemma~\ref{cf}, 
\begin{equation} \label{limlambda}
\lim_{\sigma \rightarrow 0}  \frac{\mu_u^2(\sigma)}{\sigma^2} =
\lim_{\sigma \rightarrow 0}  \frac{\displaystyle{\int_{\partial \D_\sigma} d^2(u,u(0)) d\theta}}{\sigma^3}=\lim_{\sigma \rightarrow 0} \frac{\displaystyle{2 \int_{\D_\sigma} \lambda_u \, dxdy}}{\alpha  \sigma^2}  =   \frac{2\pi \lambda_u(0)}{\alpha}.
\end{equation}
Therefore,
\begin{eqnarray*}
\lim_{i \rightarrow \infty} \frac{d^2 \big(u(z_{i'}), u(0)\big)}{|z_{i'}|^2}
 & = &
 \lim_{i \rightarrow \infty} \frac{\mu_u^2(\sigma_{i'})}{\sigma_{i'}^2}
 \cdot \
  \lim_{i \rightarrow \infty} \frac{\sigma_{i'}^2}{|z_{i'}|^2} \cdot  d_{\sigma_{i'}}^2 \big(u_{\sigma_{i'}}(\hat z_{i'}), u_{\sigma_{i'}}(0)\big)
 \\
 & = &  \frac{2\pi \lambda_u(0)}{\alpha}  \cdot  \frac{\delta^2 \big(v_*(\hat z_*), v_*(0)\big)}{|\hat z_*|^2}
= \frac{\lambda_u(0)}{\alpha} 2^{2(1-\alpha)}.
 \end{eqnarray*}
 If $\alpha>1$, then from the monotonicity property of energy (cf.~\cite[Corollary 6.8]{BFHMSZ})
 \[
 \frac{1}{\sigma^{2\alpha}} E^u[\D_\sigma]  \leq e^{\rho^\gamma}  \frac{1}{\rho^{2\alpha}} E^u[\D_\rho], \ \ 0<\sigma<\rho< \sigma_0
 \]
for some $\sigma_0$, $\gamma>0$. We therefore conclude
 \[
\lambda_u(0) = \lim_{\sigma \rightarrow 0} \frac{1}{\pi \sigma^2}  \int_{\D_\sigma} \lambda_u \, dxdy =C \lim_{\sigma \rightarrow 0} \sigma^{2(\alpha-1)} =0.
\]
Thus, for either $\alpha=1$ or $\alpha>1$,  we obtain  $
 \lim_{i \rightarrow \infty} \frac{d^2 (u(z_{i'}), u(0))}{|z_{i'}|^2}=\lambda_u(0).
$
\end{proof}

When $\mathrm{ord}^u(z_0)=1$, we get a ``lower-Lipschitz bound" near $z_0$ which depends only on the conformal factor.

\begin{lemma}\label{lem_order_one_conv}
Let $u:\Sigma \rightarrow (S,d)$ be an almost conformal harmonic map from a Riemann surface into a CAT($\kappa$) surface and $\D$ is holomorphic disk in $\Sigma$. Suppose that $z_0 \in \D$ and $\mathrm{ord}^u(z_0)=1$. Then 
\[
\lim_{z, \zeta \rightarrow z_0} \frac{d^2(u(z),u(\zeta))}{|z-\zeta|^2} =  \lambda_u(z_0)
\] 
where $\lambda_u$ is the conformal factor of $u$ in $\D$. 
\end{lemma}

\begin{proof}
Without the loss of generality, assume $z_0=0$.
It is sufficient to show that every sequence $\frac{d(u(z_i),u(\zeta_i))}{|z_i-\zeta_i|}$ with $z_i, \zeta_i \rightarrow 0$ has a subsequence that converges to $\lambda_u(0)$.  
Let $\sigma_i=2\max \{|z_i|, |\zeta_i|\}$.  By relabeling and taking a subsequence if necessary, assume  that $\sigma_i=2 |z_i| \geq 2|\zeta_i|$.   Furthermore, choose a subsequence $\sigma_{i'} \rightarrow 0$ such that   $\hat z_{i'}=\frac{z_{i'}}{\sigma_{i'}} \rightarrow \hat z_* \in \partial \D_{\frac{1}{2}}$, $\hat \zeta_{i'}=\frac{\zeta_{i'}}{\sigma_{i'}} \rightarrow \hat \zeta_* \in \overline{\D_{\frac{1}{2}}}$ and  the sequence $\{u_{\sigma_{i'}}\}$ of blow up maps converges locally uniformly in the pullback sense to an Alexandrov  tangent map $v_*$. 
Since $\mathrm{ord}^u(0)=1$, the representation of $v_*$ given by Lemma \ref{nd_lemma} implies that $\alpha_u(0)=1=\beta(u(0))$ and $I_{u(0)} \circ v_*(z)=\frac{z}{\sqrt{2\pi}}$. Moreover the metric on the tangent cone (cf. \eqref{localmetric}) is given by $ds^2= |dz|^2$. It follows that $\delta^2(v_*(\hat z_*),v_*(\hat \zeta_*))  =\frac{|\hat z_*- \hat \zeta_*|^2}{2\pi}$.
Thus, 
\begin{eqnarray*}
\lim_{i \rightarrow \infty} \frac{d^2 \big(u(z_{i'}), u(\zeta_{i'})\big)}{|z_{i'}-\zeta_{i'}|^2} & = &
 \lim_{i \rightarrow \infty} \frac{\mu_u^2(\sigma_{i'})}{\sigma_{i'}^2}
 \cdot \
  \lim_{i \rightarrow \infty} \frac{\sigma_{i'}^2}{|z_{i'}-\zeta_{i'}|^2} \cdot  d_{\sigma_{i'}}^2 \big(u_{\sigma_{i'}}(\hat z_{i'}), u_{\sigma_{i'}}(\hat \zeta_{i'})\big)
 \\
 & = & 2\pi  \lambda_u(0) \cdot  \frac{\delta^2 \big(v_*(\hat z_*), v_*(\hat \zeta_*)\big)}{|\hat z_*-\hat \zeta_*|^2}
=  \lambda_u(0).
 \end{eqnarray*}
\end{proof}

\begin{proposition}[Uniqueness of almost conformal  harmonic homeomorphisms from $\Sp^2$]  \label{unique_mobius}
An almost conformal harmonic homeomorphism from the standard sphere $\Sp^2$ into a locally CAT($\kappa$) sphere   is uniquely determined  up to a M\"obius transformation of $\Sp^2$; i.e.~if $u, v:\Sp^2 \rightarrow (S,d)$ are almost conformal harmonic homeomorphisms from the standard sphere into a CAT($\kappa$) sphere, then $u=v \circ M$ where $M:\Sp^2 \rightarrow \Sp^2$ is a M\"obius transformation.
\end{proposition}

\begin{proof}  
Let $\D'$ be a holomorphic disk with conformal coordinates $z=x+iy$, 
$\D$ be a holomorphic disk with conformal coordinates $\zeta=\xi+i\eta$  
and $v^{-1} \circ u(\D') \subset \D$.  Let $\lambda_u$ (resp.~$\lambda_v$) be the conformal factor of $u$ (resp.~$v$) in $\D'$ (resp.~$\D$).
 We denote the restriction of $v^{-1} \circ u$ to $\D'$ as 
\[
\hat u=v^{-1} \circ u\big|_{\D'}:\D' \rightarrow \D, \ \ 
\hat u(x,y) = (\xi(x,y), \eta(x,y)).
\]
For $z_0 \in \D'$ , let  $\zeta_0=\hat u(z_0)=v^{-1} \circ u(z_0)$.   
Furthermore, we write $\zeta=\hat u(z) =v^{-1} \circ u(z)$.  
Thus, 
\[
d(v(\zeta),v(\zeta_0))=d(u(z),u(z_0))
\]
and, since $v^{-1} \circ u$ is a homeomorphism,
\[
z \rightarrow z_0 \ \Leftrightarrow \ \zeta \rightarrow \zeta_0.
\]
Assume that $\lambda_v(\zeta_0) >0$ and $\mathrm{ord}^v(\zeta_0)=1$.  By applying Lemma~\ref{dl}, 
\begin{eqnarray}
\lim_{z \rightarrow z_0}   \frac{|\hat u(z)-\hat u(z_0)|}{|z-z_0|} & = &  \lim_{z \rightarrow z_0}   \frac{|\hat u(z)-\hat u(z_0)|}{d(u(z), u(z_0))} \frac{d(u(z), u(z_0))} {|z-z_0|} 
\nonumber
\\
& = &  \lim_{z \rightarrow z_0}   \frac{|\hat u(z)-\hat u(z_0) |}{d(u(z), u(z_0))} \lim_{z \rightarrow z_0} \frac{d(u(z), u(z_0))} {|z-z_0|} 
\nonumber
\\
& = &  \lim_{\zeta \rightarrow \zeta_0}   \frac{|\zeta-\zeta_0|}{d(v(\zeta), v(\zeta_0))} \lim_{z \rightarrow z_0} \frac{d(u(z), u(z_0))} {|z-z_0|} 
\nonumber
\\
& = & \sqrt{\frac{\lambda_u(z_0)}{\lambda_v(\zeta_0)}} < \infty. \label{zzet}
\end{eqnarray}
We next prove the following.\\
\\
{\sc Claim.}  $\HH^2(u^{-1} \circ v(\Ss_v))=0$ where $\Ss_v=\{\zeta \in \D:  \lambda_v(\zeta)=0\} \cup \{\zeta \in \D: \mathrm{ord}^v(\zeta)>1\}$. \\

{\sc Proof of Claim.} On the contrary, assume that $\HH^2(u^{-1} \circ v(\Ss_v))>0$.  Since $
\dim_{\HH}(\Ss_u)=0$ where $\Ss_u=\{z \in \D':  \lambda_u(z)=0\} \cup \{z \in \D': \mathrm{ord}^u(z)>1\}$ by Lemma~\ref{zeroes} and Lemma~\ref{An}, we have that $\HH^2(u^{-1} \circ v(\Ss_v) \backslash \Ss_u)>0$. Thus, by \cite[2.10.19]{federer}, there exists a constant $C$ and a point $z_0 \in u^{-1} \circ v(\Ss_v) \backslash \Ss_u$ such that 
\[
\lim_{r \rightarrow 0} \frac{\HH^2\big((u^{-1} \circ v(\Ss_v) \backslash \Ss_u) \cap \D_r(z_0) \big)}{r^2}\geq C.
\]
By Lemma~\ref{lem_order_one_conv} (since $z_0 \notin \Ss_u$) and the above inequality, we can choose $r>0$ such that
\[
d^2(u(z), u(z')) \geq \frac{1}{2} \lambda_u(z_0) |z-z'|^2, \ \ \forall z, z' \in \D_r(z_0)\subset \D'
\]
and 
\[
\HH^2\big( (u^{-1} \circ v(\Ss_v) \backslash \Ss_u) \cap \D_r(z_0) \big) \geq \frac{Cr^2}{2}.
\]
Thus, if $\HH_d^2$ represents the Hausdorff 2-dimensional measure with respect to the distance function $d$ on $S$, then
\begin{eqnarray*}
\HH^2_d\big( (v(\Ss_v) \backslash u(\Ss_u)  )\cap u(\D_r(z_0)) \big) & \geq &  \frac{1}{2} \lambda_u(z_0)  \HH^2\big( (u^{-1} \circ v(\Ss_v) \backslash \Ss_u) \cap \D_r(z_0) \big)
\\
& \geq &  \frac{Cr^2}{4} \lambda_u(z_0)>0.
\end{eqnarray*}  
On the other hand,  with $L$ the Lipschitz constant of $v$ in $\D_R(0)$ for some $R \in (0,1)$ sufficiently large such that $v^{-1} \circ u (\D_r(z_0)) \subset \D_R(0)$ and the fact that $\dim_{\HH}(\Ss_v)=0$ by Lemma~\ref{zeroes} and Lemma~\ref{An},
\[
\HH^2_d\big( (v(\Ss_v) \backslash u(\Ss_u)  )\cap u(\D_r(z_0)) \big)  \leq \HH_d^2(v(\Ss_v) \cap u(\D_r(z_0)) \leq L^2 \HH^2( \Ss_v \cap v^{-1} \circ u(\D_r(z_0)))=0.
\]
Thus, we have arrived at a contradiction.  \hfill $\Box$ ({\sc Claim})
\\

By the above claim, (\ref{zzet}) holds for a.e.~$z_0 \in \D'$. By the Radamacher-Stepanoff Theorem (cf.~\cite[3.1.9]{federer}),  $\hat u(x,y) = (\xi(x,y), \eta(x,y))$ is differentiable almost everywhere in $\D'$
and thus for a.e.~$z_0 \in \D'$ \[
|d\hat u_{z_0}(\vec v)|=\sqrt{\frac{\lambda_u(z_0)}{\lambda_v(\zeta_0)}} \mbox{ for every unit vector } \vec v.
\]
Since the right hand side of the equality above is independent of the unit vector $\vec v$ and is not zero for a.e.~$z_0 \in \D'$ (c.f. Lemma~\ref{zeroes}), we have shown  $H_{v^{-1} \circ u }(p)=1$ for a.e.~$p \in \Sp^2$.  By \cite[Theorem 16]{gehring}, $v^{-1} \circ u$ is a M\"obius transformation.
\end{proof}

\subsection{Area versus energy}\label{coarea_section}
Using the coarea formula, we demonstrate that the two dimensional Hausdorff measure of the image of a finite energy map from a disk is always less than or equal to half its energy. Of particular interest in this paper is when equality holds. 

\begin{definition}\label{monotone_defn}
Let $h:\Sigma \to (S,d)$ be a continuous map. Then $h$ is called \emph{monotone} if $h^{-1}(P)$ is connected for every $P \in S$.
\end{definition}

\begin{lemma} \label{coarea}
If  $(X,d)$ is a complete metric space and $f:\overline \D \rightarrow X$ is a finite energy map, then the following hold:
\begin{enumerate} 
\item[(a)]  
$\HH^2(f(\D)) \leq {^dE^f}/2$.
\item[(b)] \label{item2}  $\HH^2(f(\D)) = {^dE^f}/2={^dA^f}$ if $f$ is an almost conformal monotone map (cf.~Definition~\ref{area_def}).
\item[(c)] \label{item3} If 
$\HH^2(f(\D)) = {^dE^f}/2$,
$f$ is monotone, 
and  $f(\overline \D) \subset \BB_r(q_0)$ where $\overline{\BB_r(q_0)}$ is a CAT$(\kappa$) surface, 
then $f$ is an almost conformal, injective, energy minimizing map. 
\end{enumerate}
\end{lemma}

\begin{proof} 
  {  
Let $\{A_n\}$ be the disjoint measurable subsets of $\D$  such that  $\HH^2\left (\D  \backslash \bigcup_{n=1}^\infty A_n\right)=0$
and 
$
d(f(z),f(\zeta)) \leq n|z-\zeta|, \ \forall z, \zeta \in A_n
$ (cf.~{\sc Claim} in the proof of Lemma~\ref{KSjac}).  
Fix $n$,  apply the Kuratowski isometric embedding of $X$ into a Banach space $l^\infty(X)$ of bounded functions on $X$ and then apply the Kirsbraun theorem for Banach spaces to  extend the restriction map $f|_{A_n}:A_n \rightarrow X \subset l^{\infty}(X)$ to a Lipschitz map $\hat f_n:\C \rightarrow l^{\infty}(X)$.  By \cite[Theorem 2]{kirchheim}, $\text{MD}(\hat f_n,z_0)=\text{MD}_{ap}(\hat f_n,z_0)$ exists for a.e.~$z_0$.  By
applying the  coarea formula to the Lipschitz map $\hat f_n$ (cf.~\cite[Theorem 7]{kirchheim}), we  obtain
\[
\int_X \# \{\hat f_n^{-1}(p) \cap A_n \} d\HH^2(p)=\int_{A_n}  \JJ_{\hat f_n}(z) \, dxdy.
\]
Since $\hat f_n=f$ in $A_n$,  we have 
$\text{MD}(\hat f_n,z)=\text{MD}_{ap}(f,z)$ for any density 1 point $z \in A_n$. After applying the Lebesgue density theorem,
\[
\int_X \# \{f^{-1}(p) \cap A_n  \} d\HH^2(p) = 
\int_X \# \{\hat f_n^{-1}(p) \cap A_n \} d\HH^2(p)=\int_{A_n}  \JJ_f(z) \, dxdy.
\]
Summing over $n=1,2,\dots$,  we obtain
\begin{equation} \label{cards}
\int_X \# \{f^{-1}(p)\} d\HH^2(p)  =  \int_{\D} \JJ_f(z) \, dxdy.
\end{equation} 
Let  
\begin{equation} \label{setE}
E=\{z \in \D:  \JJ_f(z) =0 \}.
\end{equation}
For a.e.~$z \in \D \backslash E$,  $|f_*(\omega)|^2(z) \neq 0$ for a.e.~$\omega \in \Sp^1$.  Thus,
 \begin{eqnarray}
 \JJ_f(z)  
 & =& \left(  \frac{1}{2\pi} \int_{\omega \in \Sp^1}  |f_*(\omega)|^{-2}(z) d\HH^1(\omega) \right)^{-1} \ \ \mbox{(by Lemma \ref{KSjac})}
 \nonumber 
 \\
 & \leq &
   \frac{1}{2\pi} \int_{\omega \in \Sp^1}  |f_*(\omega)|^2(z) d\HH^1(\omega) \ \ \ \ \mbox{(by Jensen's inequality)} \label{jensen}
   \\
   & = &  |\nabla f|^2(z)/2 \ \ \ \ \ \ \ \  \
   \ \ \ \ \ \ \ \ \ \ \ \ \ \ \ \ \ \ \ \ \  \ \ \ \ \mbox{(by {\eqref{110v})}}.
   \nonumber
\end{eqnarray}
By combining (\ref{cards}) and \eqref{jensen}, we obtain 
\begin{equation} \label{string}
\HH^2(f(\D)) \leq \int_{\D} \JJ_f(z) \, dxdy = \int_{\D \backslash E} \JJ_f(z) \, dxdy\leq \frac{1}{2} \int_{\D \backslash  E} |\nabla f|^2 dxdy \leq \frac{1}{2} \int_{\D} |\nabla f|^2 \, dxdy
\end{equation}
 which proves (a).

Next,  assume $f \in W^{1,2}(\D,X)$ is an almost conformal monotone map.  The inner product structure of $\pi$ and the conformality relation (i.e.~$\pi_{11}=\pi_{22}$ and $\pi_{12}=0$) implies that at a.e.~$z \in \D$,  $|f_*(\omega)|^2(z)=\pi_{11}(z)=|\nabla f|^2(z)/2$ for a.e.~$\omega \in \Sp^1$. 
{ In  the case when $\JJ_f(z) \neq 0$ or the case  when $\JJ_f(z)=0$,     Lemma~\ref{KSjac} implies that 
 $\JJ_f(z)  = |\nabla f|^2(z)/2
$ for a.e.~$z \in \D$.}  Thus, the right hand side of (\ref{cards}) is equal to ${^d}E^f/2={^d}A^f$  (cf. Definition~\ref{area_def} and Remark~\ref{AE}).  This then implies that   $\# \{f^{-1}(p)\} \neq \infty$ for a.e.~$p \in X$ since the left hand side of (\ref{cards}) is $<\infty$.  Since $f$ is monotone, we conclude that  $\# \{f^{-1}(p)\} =1$ for a.e.~$p \in X \cap f(\D)$.  Thus,   (\ref{cards}) implies (b).

}
{ 
Finally, assume $\HH^2(f(\D)) = {^dE^f}/2$,
$f$ is monotone and  $f(\overline \D) \subset \BB_r(q_0)$.
Since $f$ is monotone, $f(\overline \D)=\overline \Omega$  where $\Omega$ is the topological disk in $\overline{\BB_r(q_0)}$ bounded by the simple closed curve $f(\partial \D)$. Let $u$ be the energy minimizing map in  $W^{1,2}_f(\D,\BB_r(q_0))$. By the continuity of $u$ (cf.~Lemma~\ref{contbdry}),  $\overline \Omega \subset u(\overline \D)$.  Combining this with item (a),  ${^dE^f} = 2\HH^2(f(\D)) \leq 2\HH^2(u(\D)) \leq {^dE^u} $, which implies that $f=u$ is the energy minimizing map in $W^{1,2}_f(\D,\BB_r(q_0))$.  
Combining (\ref{string}) with the assumption $\HH^2(f(\D)) = {^dE^f}/2$, we conclude
\[
\#\{f^{-1}(p)\cap \D\}=1 \text{ for a.e.}~p \in f(\D), 
\ \int_E |\nabla f|^2 \, dxdy =0  \ \text{ and } \  \int_{\D} \JJ_{ f}(z) dxdy = {^dE^f}/2.
\]
The second equality above implies that $|\nabla f|^2(z)=0$ for a.e.~$z \in E$, and thus $|f_*(\omega)|^2(z)=0$ for a.e.~$\omega \in \Sp^1$ and a.e.~$z \in E$.
 The third equality implies that Jensen's inequality \eqref{jensen} must be an equality for a.e.~$z \in \D\backslash E$, and thus $\omega \mapsto |f_*(\omega)|^2(z)$ is a constant function for a.e.~$z \in \D \backslash E$. We therefore conclude that $f$ is almost conformal
 which in turn implies $f$ is non-degenerate by Lemma~\ref{nd_lemma} and discrete by
Lemma~\ref{discrete}.  Since $f$ is monotone and discrete, we conclude that $f$ is injective.
 This completes the proof of (c).}
\end{proof}

\subsection{Proof Theorem~\ref{Main3}} \label{final_section}
The strategy of the proof of Theorem~\ref{Main3} is as follows:  Using a triangulation, we construct a finite energy map which is not necessarily a homeomorphism. By Corollary \ref{cor:bubbling} and Theorem~\ref{sphere_cor}, we can find an almost conformal harmonic branched cover $u$ of $\Sp^2$. We then use $u$ to define an equivalence relation on $\Sp^2$ where $\QQ :=\Sp^2 \, /\sim $ is homeomorphic to $\Sp^2$. We use the natural projection map $\pi$ to construct a  complex atlas $\tilde \Aa$ on $\QQ$.  
The key to making this work is the following consequence of the proof of Proposition~\ref{unique_mobius}:   \emph{Given restrictions $u_1=u|_{U^{(1)}}$ and $u_2=u|_{U^{(2)}}$ of $u$ to  two connected components $U^{(1)}$ and $U^{(2)}$ of $\pi^{-1}(U)$,  the composition $u_2^{-1} \circ u_1$ is a biholomorphic map.} 
For $\text{id}$ defined such that $u = \pi \circ \text{id}$, we then use Lemma~\ref{coarea} and the local results to show that $\text{id}:\QQ \rightarrow (S,d)$ is the almost conformal harmonic homeomorphism with respect to the atlas $\tilde \Aa$.

\begin{proofMain3}
We will denote by $d_{\Sp^2}$ the induced distance function on $\Sp^2$ by the standard metric $g_{\Sp^2}$.
Since we are assuming that $S$ is homeomorphic to $\Sp^2$,  we can replace  $(S,d)$ by  $(\Sp^2,d)$   by pulling  back the distance function $d$  to $\Sp^2$ from $S$  by a homeomorphism.  

In the first three steps below, we construct a finite energy continuous map $f: \Sp^2 \rightarrow (\Sp^2,d)$ (not necessarily homeomorphic). In the fourth step, we use the map $u$ to demonstrate $\text{id}$, as defined above, is the almost conformal harmonic homeomorphism.\\
\\
{\sc Step 1}.  {\it Construct a sequence $\{\TT^n_0\}$ of triangulations on $\Sp^2$ such that each $\TT^n_0$  is a geodesic triangulation with respect to $d_{\Sp^2}$ and, for  the vertex set $\VV(\TT^n_0)$ of $\TT^n_0$,}
\begin{equation} \label{goodtriangulation}
 \max \{d_{\Sp^2}(v,v'): v,v' \in \VV(\TT^n_0) \text{ such that }v \text{ and } v' \text{ are adjacent}\} \rightarrow 0
 \ 
\text{ as $n\rightarrow \infty$.}
\end{equation}

To construct $\{\TT^n_0\}$, we start with the standard sequence of triangulations which refine the equilateral triangle inscribed in the unit disk. That is, let $\triangle$ be a (closed, two-dimensional) equilateral triangle inscribed in $\overline{\D} \subset \C$.  Let $\TT^0$ be the triangulation of $\triangle$ with only one face, namely $\triangle$ itself. Then let  $\TT^1, \TT^2, \dots$ be the sequence of triangulations of $\triangle$ defined inductively by the usual refinement; i.e.~the triangulation  $\TT^n$ is defined from $\TT^{n-1}$ by taking each face $F$ of $\TT^{n-1}$ (which is an equilateral triangle) and inscribing in it an equilateral triangle, with side length half that of $F$, and letting the four resulting equilateral triangles be faces of $\TT^n$.  

    We now transfer the triangulation $\{\TT^n\}$ to the unit disk. Let $\psi:\triangle \rightarrow \overline{\D}$ be a surjective map defined in the following manner. For $p \in \partial \triangle$, let $\psi(p) \in \partial \D$  be the point where the ray from origin through $p$ intersects the unit circle $\partial \D$. For any point on the line segment from $0$ to $p$, let $\psi(p)$ map linearly onto a line segment from $0$ to $\psi(p)$.  Then $\{\psi_*(\TT^n)\}$ is a triangulation of $\overline {\D}$.
    
Finally, we transfer the triangulation to $\Sp^2$. Let $\text{proj}^-: \overline{\D} \rightarrow \Sp^2_-=\{(x,y,z) \in \Sp^2:  z \leq 0\}$ be the restriction of the stereographic projection map $\text{proj}: \C \rightarrow \Sp^2 \backslash \{(1,0,0)\}$ and let $A:\Sp^2 \rightarrow \Sp^2$ be the antipodal map $A(x,y,z)=(-x,-y,-z)$.  Define the triangulation $\TT^n_0$ on $\Sp^2$ as follows: First, we   
 push forward the vertex set  $\VV(\psi_*(\TT^n))$ of $\psi_*(\TT^n)$ along with the adjacency relation  to $\Sp^2$ via $\text{proj}^-$ and via $A \circ \text{proj}^-$.   The new vertex set is the vertex set $\VV(\TT^n_0)$ of $\TT^n_0$.  Define the edge set $\EE(\TT^n_0)$ of $\TT^n_0$ to be the set of geodesics with respect to $d_{\Sp^2}$ between $v, v' \in \VV(\TT^n_0)$ whenever $v$ and $v'$ are adjacent.  (Note that we identify the vertices and edges that overlap on the equator $\{(x,y,z) \in \Sp^2:  z=0\}$.)  Thus,  $\TT^n_0$ is a geodesic triangulation of $\Sp^2$ with respect to $d_{\Sp^2}$.  Since 
 $\text{proj}^- \circ \psi$ and  $A \circ \text{proj}^- \circ \psi$  are Lipschitz maps, we have (\ref{goodtriangulation}).\\
 \\
 {\sc Step 2}. {\it  Show that }
 \begin{equation} \label{goodtriangulation2}
\max \{d(v,v'): v,v' \in \VV(\TT^n_0) \text{ such that } v \text{ and } v' \text{ are adjacent}\} \rightarrow 0
 \ 
\text{ as $n\rightarrow \infty$.}
\end{equation}

The claim (\ref{goodtriangulation2})  follows from the fact that  the metric topology induced by $d$ is equivalent to the surface topology of $\Sp^2$ (which is in turn equivalent to the metric topology induced by $d_{\Sp^2}$).   Indeed, assume on the contrary that there exists $\epsilon>0$,  an increasing sequence $\{n_i\} \in \N$ and $v_{n_i},v_{n_i}' \in \VV(\TT^{n_i}_0)$ such that $d(v_{n_i},v_{n_i}')\geq \epsilon$ and  $v_{n_i}$ adjacent to $v_{n_i}'$.  By taking a subsequence if necessary, we can assume  
$\{v_{n_i}\}$, $\{v'_{n_i}\}$ are  converging, i.e.~$v_{n_i} \rightarrow v_\infty$ and $v'_{n_i} \rightarrow v'_\infty$.  Thus, $d(v_\infty,v'_\infty) \geq \epsilon$.  
By the equivalence of the metric topology induced by $d$ and by   $d_{\Sp^2}$, there exists a geodesic ball $\BB^{d_{\Sp^2}}_\delta(v_\infty) \subset \BB^d_\epsilon(v_\infty)$.    This is a contradiction since $v'_\infty \notin \BB^d_\epsilon(v_\infty)$ but  $v'_\infty \in  \BB^{d_{\Sp^2}}_\delta(v_\infty)$ for sufficiently large $n \in \N$.\\
\\
{\sc Step 3.}  {\it Define a finite energy map $f: \Sp^2 \rightarrow (\Sp^2,d)$.}\\

To define $f$, observe that by \eqref{goodtriangulation2} and the equivalence of the metric topologies, for $n \in \N$ sufficiently large,  each face of the triangulation $\TT^n_0$ is contained in a closed geodesic ball (with respect to $d$)   which is a CAT($\kappa$) space.   
Fix such $n \in \N$.  Let $F$ be a (closed) face of $\TT^n_0$ and $T$ be a geodesic triangle with respect to $d$ with the same vertices as $F$.  Let  $f_F: {\partial F} \rightarrow \partial T$ be a constant speed parameterization (with respect to  $d_{\Sp^2}$ on $\partial F$ and $d$ on $\partial T$) with speed $L_F$.  By Reshetnyak's theorem \cite{reshetnyak3}, we can extend this boundary parameterization to a map $f_F:F \rightarrow T$ with Lipschitz bound of  $L_F$.  (More simply, we can define $f_F$ by fixing a vertex $v_0$ and an edge $E$ opposite of $v_0$ in $F$. We then extend $f_F$ by mapping the  line from $v_0$ to a point $p \in E$  to the geodesic from $f_F(v_0)$ to $f_F(p)$ by a constant speed parameterization. By the CAT($\kappa$) condition, the extension map has a Lipschitz bound of $L_F$.)   Finally, define a Lipschitz map $f:(\Sp^2,d_{\Sp^2}) \rightarrow (\Sp^2,d)$ by setting $f|_F=f_F$, which has a  Lipschitz bound of $L=\max \{L_F:  F \in \TT^n_0\}$.    Thus, $f:\Sp^2 \rightarrow (\Sp^2,d)$ is a finite energy map, although $f$ is not necessarily a homeomorphism.  Indeed,  it is possible that the images under $f$ of  two distinct open faces of $\TT^n_0$ intersect. \\
\\
{\sc Step 4.}
{\it Use  the analysis of almost conformal harmonic maps to construct an almost conformal harmonic homeomorphism.}\\
\\   
With the finite energy map $f:\Sp^2 \rightarrow (\Sp^2,d)$ as constructed above, we can apply Corollary~\ref{cor:bubbling} to assert the existence of a  harmonic map $u:\Sp^2 \rightarrow (\Sp^2,d)$.
By Theorem~\ref{sphere_cor},  $u$ is an almost conformal  branched cover.  Denote the  branch set of $u$ by $\BB$.  If $u$ is injective, then $u$ is an almost conformal harmonic homeomorphism.  This completes {\sc Step 4}, so we will   assume instead that $u|_{\Sp^2 \backslash \BB}$ is a $k$-sheeted cover of $\Sp^2 \backslash u(\BB)$ for $k>1$.

Define an equivalence relation on $\Sp^2$ by setting
\[
p \sim q \ \ \Leftrightarrow \ \ u(p) = u(q).
\]
Denote the quotient space $\Sp^2 \, /\sim$ by $\QQ$; i.e.~$\QQ$ is the set of equivalence classes $[ \cdot ]$. The topology on $\QQ$ is defined by requiring that $U \subset \QQ$ is open if and only if $\pi^{-1}(U)$ is open where 
\[
\pi:  \Sp^2 \rightarrow \QQ, \ \ \pi(p)=[p]
\]
is the natural projection map. 
The induced map $[p] \mapsto u(p)$  is essentially the identity map of $\Sp^2$ and thus we will denote it as 
\[
\text{id}:  \QQ  \rightarrow (\Sp^2,d).
\]Since $\id$ is a closed, continuous bijection,  $\QQ$ is a topological sphere.
In summary, we have the following commutative diagram:
\[
  \begin{tikzcd}
    \Sp^2 \arrow{d}[swap]{\pi} \arrow[swap]{dr}[swap]{u} & \\
    \arrow{r}[swap]{\text{id}} \QQ 
     & (\Sp^2,d) 
  \end{tikzcd}
\]  
The restriction $\pi|_{\Sp^2 \backslash \BB}$ is a $k$-sheeted cover of $\QQ \backslash \pi(\BB)$.

Let $\Aa$ be a complex structure on $\QQ \backslash \pi(\BB)$ which  makes $\pi|_{\Sp^2 \backslash \BB}$ a holomorphic covering map.  More precisely, we can define $\Aa$ as  follows: For $[p] \in \QQ \backslash \pi(\BB)$, let

\begin{itemize} 
\item  $U_{[p]}$  be a neighborhood  of $[p]$ in $\QQ \backslash \pi( \BB)$,
\item $\left\{U^{(i)}_{[p]}\right\}_{i=1, \dots, k}$  be the disjoint open sets of $\Sp^2$ such that $\bigcup_i U^{(i)}_{[p]} = \pi^{-1}(U_{[p]})$ and  $\pi\big|_{U^{(i)}_{[p]}}: U^{(i)}_{[p]} \rightarrow U_{[p]}$ is a homeomorphism, and
\item $\{\varphi^{(i)}_{[p]}: U^{(i)}_{[p]} \rightarrow \D \subset \C\}$ be complex charts  of $\Sp^2$.
\end{itemize}
The key observation we will use below is that 
\begin{equation} \label{ishol}
\left( \pi|_{U_{[p]}^{(i)}} \right)^{-1} \circ \pi|_{U_{[p]}^{(1)}} = \left( u|_{U_{[p]}^{(i)}} \right)^{-1} \circ u|_{U_{[p]}^{(1)}}
\mbox{ 
is a biholomorphic map.}
\end{equation}  
The validity of  (\ref{ishol}) follows from the fact that the right hand side of the equation can be shown to be holomorphic by the same argument as the proof of Proposition~\ref{unique_mobius} (cf.~\cite[Theorem 16]{gehring}).

For each $i=1, \dots, k$, define
\[
\bar \varphi^{(i)}_{[p]}:=\varphi^{(i)}_{[p]} \circ \left(\pi|_{U_{[p]}^{(i)}}\right)^{-1}:  U_{[p]} \rightarrow \D.
\]
We claim that the atlas 
\[
\Aa=\left\{\left(U_{[p]}, \bar \varphi^{(1)}_{[p]}\right)\right\}_{[p] \in \QQ \backslash \pi(\BB)}
\]
covering $\QQ \backslash \pi(\BB)$ defines  a complex structure on $\QQ \backslash \pi(\BB)$.  To see this, first note that 
\begin{equation} \label{bhm}
\bar \varphi^{(i)}_{[p]} \circ \left(\bar \varphi^{(1)}_{[p]}\right)^{-1} :\D \rightarrow \D \  \mbox{ is biholomorphic.}
\end{equation}
Indeed, (\ref{bhm}) follows from (\ref{ishol}) and the fact that
\begin{eqnarray*}
\bar \varphi^{(i)}_{[p]} \circ \left(\bar \varphi^{(1)}_{[p]}\right)^{-1}& = & 
\varphi^{(i)}_{[p]} \circ \left(\pi|_{U_{[p]}^{(i)}}\right)^{-1} \circ \left(\varphi^{(1)}_{[p]} \circ \left(\pi|_{U_{[p]}^{(1)}}\right)^{-1}\right)^{-1} 
\\
& = & \varphi^{(i)}_{[p]} \circ \left(\pi|_{U_{[p]}^{(i)}}\right)^{-1} \circ  \pi|_{U_{[p]}^{(1)}} \circ \left(\varphi^{(1)}_{[p]}\right)^{-1}.
\end{eqnarray*}

 If $U_{[p]} \cap U_{[q]} \neq \emptyset$, then there exists $i, j$ such that $U_{[p]}^{(i)} \cap U_{[q]}^{(j)} \neq \emptyset$.  Since 
 \[
 \bar \varphi_{[p]}^{(i)} \circ \left(  \bar \varphi_{[q]}^{(j)}\right)^{-1}
 =
 \varphi_{[p]}^{(i)} \circ \left(  \varphi_{[q]}^{(j)}\right)^{-1}
\mbox{ on $\varphi_{[q]}^j\left(U_{[p]}^{(i)} \cap U_{[q]}^{(j)}\right)$}
\]
and
\[
\bar \varphi_{[p]}^{(1)} \circ \left(\bar \varphi_{[q]}^{(1)} \right)^{-1} = \bar \varphi_{[p]}^{(1)} \circ \left(\bar \varphi_{[p]}^{(i)}\right)^{-1} \circ  \varphi_{[p]}^{(i)}      \circ \left(\varphi_{[q]}^{(j)} \right)^{-1}\circ  \bar \varphi_{[q]}^{(j)}\circ \left(\bar \varphi_{[q]}^{(1)}\right)^{-1},
\] 
we conclude 
\begin{equation}\label{transition1}
\bar \varphi_{[p]}^{(1)} \circ \left(\bar \varphi_{[q]}^{(1)} \right)^{-1}: \ \bar \varphi_{[q]}^{(1)}\left(U_{[p]}\cap U_{[q]}\right) \ \rightarrow \ \bar \varphi_{[p]}^{(1)}\left(U_{[p]}\cap U_{[q]}\right) \mbox{ is biholomorphic}.
\end{equation}
We have thus shown that the transition maps of $\Aa$ are holomorphic as required, and hence $\Aa$ is a complex atlas.

For any $[b] \in \pi(\BB)$, let  $U_{[b]}$ be a   neighborhood of $[b]$ such that a connected component $\UU \subset \Sp^2$ of $\pi^{-1}(U_{[b]})$   satisfies $\UU \cap \BB = \{b\}$ and $\UU^*:=\UU \backslash \{b\}$ is biholomorphic to $\D^*:=\D \backslash \{0\}$.  Let $U_{[b]}^*:=\pi(\UU^*)=U_{[b]} \backslash \{[b]\}$ and define
\[
\Aa_{[b]}=\left \{ \left(U_{[p]} \cap U_{[b]}^*, \bar \varphi_{[p]}^{(1)}|_{U_{[p]} \cap U_{[b]}^*}\right):\left(U_{[p]}, \bar \varphi_{[p]}^{(1)} \right)\in \Aa\right\}.
\]
In other words, $\Aa_{[b]}$ is the restriction of the complex charts of $\Aa$ to $U_{[b]}^*$,   and hence
defines a complex structure on $U_{[b]}^* \subset \QQ$.
Since $U_{[b]}^*$ is homeomorphic to an annulus and  $\pi|_{\UU^*}:\UU^* \simeq \D^* \rightarrow U_{[b]}^*$ is a holomorphic covering map with respect to the complex charts $\Aa_{[b]}$,  the Riemann surface $(U_{[b]}^*, \Aa_{[b]})$ is biholomorphic to $\D^*$; i.e.~there exists a homeomorphism  
  \[
  \phi_{[b]}:U_{[b]}^* \rightarrow \D^*
  \]
such that, for any chart  $ \left(U_{[p]} \cap U_{[b]}^*, \bar \varphi_{[p]}^{(1)}|_{U_{[p]} \cap U_{[b]}^*}\right) \in \Aa_{[b]}$,
\begin{equation} \label{phib}
\phi_{[b]} \circ   (\bar \varphi_{[p]}^{(1)})^{-1}: 
\bar   \varphi_{[p]}^{(1)}\left(U_{[p]} \cap U^*_{[b]}\right) \   \rightarrow \  \phi_{[b]} \left(U_{[p]} \cap U^*_{[b]}\right)\mbox{is biholomorphic.}
  \end{equation}
  Extend $\phi_{[b]}$ to a homeomorphism 
   \[
   \bar \phi_{[b]}: U_{[b]} \rightarrow \D.
   \]  
The atlas $\widetilde{\Aa}:=\Aa \cup \left\{\left(U_{[b]}, \bar \phi_{[b]}\right)\right\}_{b \in \pi(\BB)}$ defines a complex structure  which makes $\QQ$ into  the Riemann sphere $\Sp^2$.  Indeed, \eqref{transition1} and (\ref{phib}) show that the transition maps in $\tilde \Aa$ are biholomorphic.  

With the complex structure on $\QQ$ defined by $\widetilde{\Aa}$, the homeomorphism 
\[
\text{id}: \Sp^2 \simeq \QQ \rightarrow (\Sp^2,d)
\]
 is an almost conformal harmonic map. To see this, first note that $\pi|_{U_{[p]}^{(1)}}$ is a biholomorphic map in the coordinate neighborhood $U_{[p]}$  for any $p \notin \BB$ and  $u=\text{id} \circ \pi$ is an almost conformal harmonic map. Therefore, $\text{id}$ is an almost conformal harmonic map in $\QQ \backslash \pi(\BB)$.   Thus, by the removable singularities theorem (cf.~\cite[Section 3]{paper2}), $\text{id}$ is an almost conformal harmonic map on $\Sp^2 \simeq \QQ$.  Lemma~\ref{coarea} (b) and  Theorem~\ref{sphere_cor} imply that $\text{id}$ is a conformal harmonic homeomorphism satisfying $\HH^2(\text{id}(\Sp^2)) = {^dE^{\text{id}}[\Sp^2]}/2$.  Uniqueness follows from Proposition~\ref{unique_mobius}.
\end{proofMain3}

\appendix

\section{The order function}\label{App_order}
We use the notation of \cite{BFHMSZ}, which differs slightly from notation within this paper. The interested reader will have an easier time checking the details of the proofs herein as they relate to the work in \cite{BFHMSZ}. Note that the role of $\sigma_j\to 0$ in this paper is replaced by $\lambda_j \to 0$ below. Also $\mathcal CX$ denotes the NPC cone over $X$.
\begin{lemma}\label{order_equal}
Let $u:B_1(0) \to (X,d)$ be a finite energy harmonic map where $B_1(0)\subset M$, $(M,g)$ is a Riemannian manifold and $(X,d)$ is a locally CAT($\kappa$) space. Let $x \in M$ and let $\overline u_*:B_1(0) \to (X_*,d_*)$ denote a tangent map of $u$ at $x$ as constructed in \cite[Proposition 7.5]{BFHMSZ}. Then
\[
\mathrm{ord}^{u_*}(0)=\alpha_*(0) = \alpha(x)= \mathrm{ord}^u(x).
\]
\end{lemma}
\begin{proof}
By the proof of \cite[Lemma 8.1]{BFHMSZ}, it is enough to show that 
\begin{equation}\label{energy_convergence}
{^{d_*}}E^{\overline u_*}(\sigma) = \lim_{k \to \infty} \, {^{D_k}}E^{\overline u_k}_{g_k}(\sigma)
\end{equation}where 
\begin{align*}
u_k(x)&:= u(\lambda_k x)\\
\overline u_k(x)&:= [u_k(x),1] \in X \times \{1\} \subset \mathcal CX\\
g_k(x)&:= g(\lambda_k x)\\
d_k(x)&:= (\lambda_k^{1-n} I_k)^{-1/2}d(p,q)\\
D_k(p,q)&:= (\lambda_k^{1-n} \overline I_k)^{-1/2}D(p,q)\\
I_k&:= \inf_{q \in X}\int_{\partial B_{\lambda_k}(0)}d_k^2(u,q) d\Sigma_g\\
\overline I_k&:= \inf_{q \in \mathcal CX}\int_{\partial B_{\lambda_k}(0)}D_k^2([u,1],q) d\Sigma_g.
\end{align*}By \cite[Proposition 7.5]{BFHMSZ}, $\overline u_k$ converges locally uniformly in the pullback sense to $\overline u_*$. Since $\overline u_k$ maps into the NPC metric space $(\mathcal CX, D_k)$, it suffices to prove that $\overline u_k$ satisfies the hypotheses of \cite[Theorem 3.11]{korevaar-schoen2} when we consider $\overline u_k$ defined on a domain with a fixed metric. (The conclusion of this theorem gives the convergence of energy density measures.) To that end, we prove the following two claims for $\overline u_k:(B_1(0), \delta) \to (\mathcal CX, D_k)$, where $\delta$ is the Euclidean metric.

\begin{claim}For $k$ large enough, $\overline u_k$ is within $\epsilon_k$ of minimizing on $(B_1, \delta)$ with $\lim_{k \to \infty} \epsilon_k=0$.
\end{claim}
\begin{proof}
Let $v_k= {^{Dir}}\overline u_k:B_1 \to (\mathcal CX, D_k)$ be the Dirichlet solution for $\overline u_k$, but with respect to the Euclidean metric $\delta$. We normalize the metric $g_k$ (and continue to refer to it as $g_k$ for convenience) and we recall that the normalization preserves energy. In particular, $u_k$ is still minimizing with respect to the normalization. Since $g_k$ is smooth, there exists $c>0$ such that for all $\Omega \subset B_1$, 
\begin{equation}\label{deltavsg}
 (1-c\lambda_k) {^{D_k}}E^{\overline u_k}_{\delta}[\Omega]\leq {^{D_k}}E^{\overline u_k}_{g_k}[\Omega] \leq (1+c\lambda_k) {^{D_k}}E^{\overline u_k}_{\delta}[\Omega].
\end{equation}Note that the same string of inequalities holds for $v_k$ as well. 
It follows that 
\begin{align*}
{^{D_k}}E^{\overline u_k}_{\delta}[B_1] &\leq (1-c\lambda_k)^{-1}{^{D_k}}E^{\overline u_k}_{g_k}[B_1] \mbox{ by } \eqref{deltavsg}\\
& \leq  (1-c\lambda_k)^{-3}{^{D_k}}E^{v_k}_{g_k}[B_1] \mbox{ by the proof of \cite[Proposition 7.5]{BFHMSZ}} \\
&\leq (1-c\lambda_k)^{-3}(1+c\lambda_k) {^{D_k}}E^{v_k}_{\delta}[B_1] \mbox{ by } \eqref{deltavsg}\\
& \leq (1+ C\lambda_k){^{D_k}}E^{v_k}_{\delta}[B_1]\\
& \leq {^{D_k}}E^{v_k}_{\delta}[B_1] + 2C\lambda_k {^{d_*}}E^{\overline u_*}[B_1]\mbox{ by the proof of \cite[Proposition 7.5]{BFHMSZ}}.
\end{align*}
\end{proof}

\begin{claim}There exists $C>0$ independent of $t>0$ and of $k$ such that
\[
{^{D_k}}E^{\overline u_k}_{\delta}(B_1 \backslash B_{1-t}) \leq Ct.
\]

\end{claim} 
\begin{proof}
Note that for $k$ large enough, $u|_{B_{2\lambda_k}}$ is minimizing and therefore $u_k|_{B_{2}}$ is minimizing. 
By \cite[Proposition 8.2]{BFHMSZ} and the proof of \cite[Lemma 7.5]{BFHMSZ}, there exists a constant $C'>0$ independent of $k$ such that for $x,y \in B_{3/2}$, 
\[
D_k(\overline u_k(x),\overline u_k(y))\leq 2  d_k(u_k( x), u_k( y)) \leq C'|x-y| .
\]It follows that $\overline u_k$ is Lipschitz on $B_{3/2}$ with constant $C'$ independent of $k$. Therefore
\[
{^{D_k}}E^{\overline u_k}_{g_k}(B_1 \backslash B_{1-t}) \leq Ct.
\] where $C$ depends only on $C'$ and the dimension of $M$. The result follows with $\delta$ in place of $g_k$ by the estimate \eqref{deltavsg}. 

\end{proof}The two claims imply that $\overline u_k$ satisfy the hypotheses of \cite[Theorem 3.11]{korevaar-schoen2} and thus 
\[
{^{d_*}}E^{\overline u_*}(\sigma) = \lim_{k \to \infty} \, {^{D_k}}E^{\overline u_k}_{\delta}(\sigma).
\]Applying \eqref{deltavsg} then implies \eqref{energy_convergence}.
\end{proof}

\begin{lemma}\label{order_upper}Let $(M,g)$ be a Riemannian manifold and $(X,d)$ be a locally compact CAT($\kappa$) space. Let $u:(M,g) \to (X,d)$ be a harmonic map. Then
the order function $\alpha_u$ is upper semi-continuous.
\end{lemma}
\begin{proof}
It is enough to show that $\alpha_u$ is the decreasing limit of continuous functions. By definition,
\[
\alpha_u(x)=\lim_{\sigma \to 0^+} \frac{\sigma E_x(\sigma)}{I_x(\sigma,Q_\sigma)}= \lim_{\sigma \to 0^+} \frac{\sigma F_x(\sigma)}{I_x(\sigma,Q_\sigma)} \lim_{\sigma \to 0^+} \frac{E_x(\sigma)}{F_x(\sigma)}
\]where $E_x(\sigma), F_x(\sigma), I_x(\sigma, Q_\sigma)$ are defined in \cite[Section 6]{BFHMSZ} and the subscript ``$x$" signifies the centering of each ball at $x\in M$. By \cite[Lemma 6.3]{BFHMSZ}, $ \lim_{\sigma \to 0^+} \frac{E_x(\sigma)}{F_x(\sigma)}=C_1<\infty$ and thus
\[
\alpha_u(x)=C_1\lim_{\sigma \to 0^+} \frac{\sigma F_x(\sigma)}{I_x(\sigma,Q_\sigma)}.
\]In \cite[Section 6]{BFHMSZ} it is verified that $\frac{\sigma F_x(\sigma)}{I_x(\sigma,Q_\sigma)}$ is monotone nondecreasing in $\sigma$. Therefore, $\alpha_u$ is upper semi-continuous.
\end{proof}

\section{Proof of Proposition~\ref{coneovercurve}} \label{appB}

\begin{proof}
Let $r>0$ be as in Remark~\ref{3properties}.
The strategy is to first show  that $\partial \BB_\epsilon(q_0)$ is homeomorphic to a circle for $\epsilon \in (0,r)$.  Using this, we will then show that the space of directions $\EE_{q_0}$ is homeomorphic to a circle. Local compactness then implies that $\EE_{q_0}$ is isometric to a simple closed curve of finite length.

We can assume that $r>0$ is chosen sufficiently small such that  there exists a homeomorphism 
 $
 h: \D \rightarrow \BB_r(q_0)$.  Without the loss of generality, we may assume $h(0)=q_0$. Fix $\epsilon \in (0,r)$.     Let 
 \[
 \overline \B_\epsilon  =h^{-1}(\overline{\BB_\epsilon(q_0)})
 \]
and
\[
\delta=\inf_{x \in \partial \D, y \in \overline  \B_\epsilon} |x-y|.
\]
Since  $\partial \D$ and $\overline  \B_\epsilon$ are compact sets, $\delta>0$.  Thus,  $\D_{1-\frac{\delta}{2}}$  contains $\overline  \B_\epsilon$.
Pull back the distance function on $S$ to $\D$ via $h$ and still denote it by $d$.  Thus, $(\D,d)$ has the same properties (i.e.~uniqueness, continuity and extendability of geodesics) as $\BB_r(q_0)$.    Throughout  this proof, we adopt the following notation:
 \begin{itemize}
 \item $\gamma_Q$ is the geodesic from 0 to a point $Q$.
 \item $\gamma_{PQ}$ is the geodesic from a point $P$ to a point $Q$. 
 \end{itemize}
 
We will prove that $\partial \B_{\epsilon}$ is homeomorphic to a circle by showing that $\partial \B_\epsilon$ is path connected and $\partial \B_\epsilon \backslash \{q_1,q_2\}$ is disconnected for any $q_1,q_2 \in \partial \B_\epsilon$, $q_1 \neq q_2$.  Indeed, these two properties  characterize $\partial  \B_\epsilon$ as a topological circle by ~\cite{moore}.  \\
\\
$\bullet$  {\sc Proof that $\partial \B_\epsilon$ is connected.}\\

The nearest point projection (with respect to the metric $d$)
\[
\pi_\epsilon: \partial \D_{1-\frac{\delta}{2}} \rightarrow \partial \B_\epsilon=h^{-1}(\partial \BB_\epsilon(q_0))
\]
is well-defined.  Indeed, for any $Q \in \partial \D_{1-\frac{\delta}{2}}$, the unique geodesic $\gamma_Q$ from $0$ to $Q$ intersects a unique point in $\partial \B_\epsilon$.  
The map $\pi_\epsilon$ is continuous by  the CAT($\kappa$) property (cf.~\cite[II.1.7]{bridson-haefliger}).  For $q \in \partial \B_\epsilon$, property (i) implies that the    geodesic $\gamma_q$  can be extended to a  geodesic $\gamma_Q$ for  $Q \in \partial \D_{1-\frac{\delta}{2}}$.   This in turn implies that $\pi_\epsilon(Q)=q$, thereby proving $\pi_\epsilon$ is surjective.  Let $q_1,q_2 \in \partial \B_\epsilon$.  By surjectivity of $\pi_\epsilon$, there exist $Q_1,Q_2 \in \partial \D_{1-\frac{\delta}{2}}$ such that $\pi_\epsilon(Q_1)=q_1$ and $\pi_\epsilon(Q_2)=q_2$.  Let $\bar A\subset \partial \D_{1-\frac{\delta}{2}}$ be a closed arc connecting $Q_1$ to $Q_2$.  By the continuity of $\pi_\epsilon$,  $\pi_\epsilon(\bar A)$ is a path from $q_1$ to $q_2$.  This proves $\partial \B_\epsilon$ is path connected.\\
\\
$\bullet$ {\sc  Proof that $\partial \B_\epsilon \backslash \{q_1,q_2\}$ is disconnected for $q_1, q_2 \in \partial \B_\epsilon$, $q_1 \neq q_2$.} \\

For $i=1,2$, extend the geodesic $\gamma_{q_i}$ to a geodesic $\gamma_{Q_i}$ with $Q_i \in \partial \D_{1-\frac{1}{2}}$.  Let $q$ be the point on $\gamma_{Q_1}  \cap \gamma_{Q_2}$ furthest away from $0$.  
Thus, $\gamma_{q} \cup \gamma_{qQ_i} = \gamma_{Q_i}$ and 
\begin{equation} \label{ivt}
d(0,q) < \epsilon<d(0,Q_1)
\end{equation} since $\pi_\epsilon(Q_1)=q_1 \neq q_2 =\pi_\epsilon(Q_2)$ and $d(0,q_1)=\epsilon$.   Let $A$ and $A'$ be the  two distinct open arcs of $\partial \D_{1-\frac{\delta}{2}} \backslash \{Q_1,Q_2\}$.  The simple closed  curve $\Gamma=\gamma_{qQ_1} \cup \gamma_{qQ_2}\cup  A$ bounds a topological disk which we denote by $\DD$.   (Note that from Schoenflies theorem,  there exists a homeomorphism $f:\R^2 \rightarrow \R^2$ such that $f(\Gamma)$ is the unit circle in $\R^2$.) Similarly, the simple closed  curve  $\Gamma'=\gamma_{qQ_1} \cup \gamma_{qQ_2}\cup A'$ bounds another topological disk which we denote by $\DD'$.  

Contrary to the claim, assume $\partial \B_\epsilon \backslash \{q_1,q_2\}$ is connected.  Then either $\partial \B_\epsilon \backslash \{q_1,q_2\} \subset \DD$ or $\partial \B_\epsilon \backslash \{q_1,q_2\} \subset \DD'$.  Relabeling if necessary, assume the latter which implies that $\partial \B_\epsilon \cap \DD=\emptyset$.  On the other hand, $\DD$ is a topological disk, and thus we can choose a curve $\sigma$ from $q \in \Gamma = \partial \DD$ to $Q_1 \in \Gamma= \partial \DD$ whose interior is contained in $\DD$.  By the intermediate value theorem and (\ref{ivt}), there exists  $p \in \sigma$ such that $d(0,p)=\epsilon$.  Thus, $p \in \partial \B_\epsilon \cap \DD$ which is a  contradiction to the fact that $\partial \B_\epsilon \cap \DD = \emptyset$.   This proves $\partial \B_\epsilon \backslash \{q_1,q_2\}$ is disconnected. \\

Now that we have shown $\partial \B_\epsilon=h^{-1}(\partial \BB_\epsilon(q_0))$ is a topological circle, we will use this fact  to show  $\EE_{q_0}$ is also a topological circle.  We will use the map
\[
\LL:  \partial \B_\epsilon \rightarrow \EE_{q_0}, \ \ Q \mapsto [\gamma_Q]
\]  
to show that $\EE_{q_0}$ is connected and $\EE_{q_0} \backslash \{[\gamma_{Q_1}], [\gamma_{Q_2}]\}$ is disconnected if  $[\gamma_{Q_1}] \neq [\gamma_{Q_2}]$.\\
\\
$\bullet$ {\sc Proof that $\EE_{q_0}$ is connected.}
\\
\\
The map $\LL$ is continuous and surjective by  properties (ii) and (iii) which implies that $\EE_{q_0}$ is path connected.  Indeed, let $\sigma$ be a path from $Q_1$ to $Q_2$ in $\partial \BB_\epsilon$.  Then $\LL(\sigma)$ is a path from $[\gamma_{Q_1}]$ to $[\gamma_{Q_2}]$.
\\
\\
$\bullet$ {\sc Proof that $\EE_{q_0} \backslash \{[\gamma_{Q_1}], [\gamma_{Q_2}]\}$ is disconnected if $[\gamma_{Q_1}] \neq [\gamma_{Q_2}]$.}\\
\\
This proof will consist of two steps. First,  we will show that $\LL$ is a monotone map; i.e.~$\LL^{-1}([\gamma_{Q}])$ is connected for any $[\gamma_{Q}] \in \EE_{q_0}$ (cf.~{\sc Claim 1}).  This implies that $\partial \B_\epsilon \backslash (\LL^{-1}([\gamma_{Q_1}]) \cup \LL^{-1}([\gamma_{Q_2}]))=U \cup U'$ where $U,U'$ are distinct open sets.  Second, we will prove that $\LL(U)$ and $\LL(U')$ are open arcs (cf.~{\sc Claim 2}). This proves   $\EE_{q_0} \backslash \{[\gamma_{Q_1}], [\gamma_{Q_2}]\}$ is disconnected.  \\
\\
{\sc Claim 1.} $\LL$ is  monotone.  \\
\\
{\sc Proof}.   
Let $[\gamma_{Q}] \in \EE_{q_0}$ with $Q \in \partial \B_\epsilon$,  and let $P \in \LL^{-1}([\gamma_{Q}])$ with $Q \neq P$.  Thus, $\angle (\gamma_{Q}, \gamma_{P})=0$.  Since $\partial \B_\epsilon$ is a topological circle, there exist exactly  two distinct  connected open arcs of $\partial \B_\epsilon \backslash \{Q,P\}$ which we call $A$ and $A'$. 
  It is sufficient to show that  one of $\bar A$ or $\overline {A'}$ is contained in $\LL^{-1}([\gamma_{Q}])$.  We do this by letting $q$ be the point on $\gamma_{Q} \cap \gamma_{P}$ furthest away from $0$ and considering the following two cases separately.  \\

{\sc Case 1}:  $q \neq 0$.  Let $\DD$ be the  topological disk  bounded by the simple closed curve  $\gamma_{qQ} \cup \gamma_{qP} \cup A$, and let $\DD'$ be the topological disk  bounded by  the simple closed curve  $\gamma_{qQ} \cup  \gamma_{qP} \cup A'$.  By relabeling if necessary, we will assume that  $\gamma_{q} \subset \overline{\DD'}$ and $\gamma_{q} \cap \DD=\emptyset$.   
For $Q' \in \bar A \subset \partial \B_\epsilon$, we observe that  $\gamma_{Q'}=\gamma_{q} \cup \gamma_{qQ'}$ with $\gamma_{qQ'} \subset \overline \DD$.   Since $ \gamma_{q} \subset \gamma_Q \cap \gamma_{Q'}$, we conclude $\angle(\gamma_Q, \gamma_{Q'})=0$.  In other words, $\gamma_{Q'} \in [\gamma_{Q}]$ which implies $[\gamma_{Q'}] \in \LL^{-1}([\gamma_{Q}])$ for all $Q' \in \bar A$.  This implies that $\bar A \subset \LL^{-1}([\gamma_{Q}])$ and $\bar A$ is a path from $Q$ to $P$ in $\LL^{-1}([\gamma_{Q}])$.\\

{\sc Case 2}: $q=0$. For $t \in (0,\epsilon]$, let $t \mapsto Q(t)$ (resp.~$P(t)$) be the arclength parameterization of $\gamma_Q$ (resp.~$\gamma_P$).  The assumption that $q =0$ implies  $Q(t) \neq P(t)$ for all $t \in (0,\epsilon]$.   Let $\gamma_t$ be the geodesic from $Q(t)$ to $P(t)$.  Fix $t_0 \in (0,\epsilon)$ and a point  $q_0' \in \gamma_{t_0} \backslash \{Q(t_0),P(t_0)\}$. Let $\gamma_{Q_0'}$ for $Q_0' \in \partial \B_\epsilon$ be the geodesic extension of $\gamma_{q_0'}$.  By relabeling if necessary, assume $Q_0' \in A$.  
Let $\DD$ be the topological disk bounded by the simple closed curve $\gamma_Q \cup \gamma_P \cup A$.  Then $\bar \DD$ is geodesically  convex.  

Fix any $Q' \in \bar A$ and $t \in (0,\epsilon)$.  Observe that   $\DD \backslash \gamma_t$ equals two open sets  $\DD_1$ and $\DD_2$ with $0 \in \overline{\DD_1}$ and   $Q' \in \overline {\DD_2}$.   Since  $\gamma_{Q'} \subset \overline \DD$, $\gamma_{Q'}$ must intersect $\gamma_t$.  Furthermore,  $\gamma_t$ cannot intersect $\gamma_{Q'}$ at  more than one point because of geodesic uniqueness. Thus, $\gamma_{Q'} \cap \gamma_t$ contains exactly one point which we denote by $q_t'$.

We now consider the geodesic triangle $\triangle 0Q(t)q_t'$ (with vertices $0$, $Q(t)$ and $q_t'$).  We claim that $d(Q(t),q_t')$, the  length of the side opposite to the vertex, is equal to  $o(t)$ while   $d(0,Q(t))$ and $d(0,q_t')$,  the lengths of the adjacent sides to the vertex 0, is equal to $O(t)$.  Indeed, since 
\[
0=\angle (\gamma_{Q}, \gamma_{P})= \lim_{t \rightarrow 0} \tilde \angle (Q(t),P(t)),
\]
we have
\[
\lim_{t \rightarrow 0} \frac{d(Q(t), q_t')}{t}  \leq \lim_{t \rightarrow 0} \frac{d(Q(t),P(t))}{t}=0.
\]
Furthermore, $d(0,Q(t))=t$ by definition.  Thus, by the triangle inequality,
\[
1= \lim_{t \rightarrow 0}  \frac{d(Q(t),0)-d(Q(t) ,q_t')}{t} \leq \lim_{t \rightarrow 0} \frac{d(q_t',0)}{t}.
\]
 We therefore conclude by the definition of $\tilde \angle$ that 
\[
\angle (\gamma_Q,\gamma_{Q'}) = \lim_{t \rightarrow 0} \tilde \angle (Q(t),q_t')  =0, \ \ \ \text{ for }Q' \in \bar A.
\]
This implies that $\bar A \subset \LL^{-1}([\gamma_Q])$ and $\bar A$ is a path from $Q$ to $P$ in $\LL^{-1}([\gamma_Q])$. \hfill $\Box$({\sc Claim 1})\\

 Since $\LL$ is monotone, $\LL^{-1}([\gamma_{Q_i}])$ is connected for $i=1,2$.  Thus
$\partial \B_\epsilon \backslash (\LL^{-1}([\gamma_{Q_1}]) \cup \LL^{-1}([\gamma_{Q_2}])= U \cup U'$ where $U,U'$ are distinct open sets.     \\
\\
{\sc Claim 2.}   $\LL(U)$ and $\LL(U')$ are open subsets in  $\EE_{q_0}$.   \\
\\
{\sc Proof.}
Let $[\gamma_{Q_0}] \in \LL(U)$ with $Q_0 \in U$. We will first show that   there exists $\delta>0$ such that $\angle(\gamma_{Q_0}, \gamma_P) >\delta$ for all $P \in \overline{U'}$.  Indeed, 
on the contrary, assume that there exist $P_i \in \overline{U'}$ such that $\angle(\gamma_{Q_0}, \gamma_{P_i}) < \frac{1}{i}$.  By taking a subsequence if necessary, assume $P_i \rightarrow Q_\infty \in \overline{U'}$. By the continuity of angles,
$\angle (\gamma_{Q_0}, \gamma_{Q_\infty})= \lim_{i \rightarrow \infty} \angle (\gamma_{Q_0}, \gamma_{P_i})=0$ which implies that $Q_\infty \in \LL^{-1}([\gamma_{Q_0}]) \subset U$. Since $U \cap \overline{U'}=\emptyset$, this is a contradiction.  Thus, for  $\delta>0$ as above,   
the geodesic ball $\BB^{\EE_{q_0}}_{\delta}([\gamma_{Q_0}])$ is contained in $\LL(U)$. This proves $\LL(U)$ is open.  Similarly, $\LL(U')$ is open.    \hfill $\Box$({\sc Claim 2})
\\

  Since $\LL(U) \cap \LL(U') = \emptyset$,  $\LL(U) \cup \LL(U') = \EE_{q_0} \backslash \{[\gamma_{Q_1}], [\gamma_{Q_2}]\}$ and $\LL(U)$, $\LL(U')$ are open, we conclude $\EE_{q_0} \backslash \{[\gamma_{Q_1}], [\gamma_{Q_2}]\}$ is disconnected.  
By \cite{moore}, $\EE_{q_0}$ is homeomorphic to a circle.  As a closed and bounded set, $\partial \BB_\epsilon$ is compact.  Thus, $\EE_{q_0}=\LL(\partial \BB_\epsilon)$ is also compact which implies that it is isometric to a finite length closed curve.
We therefore conclude that   $\EE_{q_0}$ isometric to a simple closed curve of finite length. 
\end{proof}

\section{Proof of Lemma~\ref{KSjac}} \label{appC}

\begin{proof}
We start with the following claim (see also \cite{hkst}):\\
\\
{\sc Claim.} {\it There exists a set $\{A_n\}$ of countable disjoint measurable subsets of $\D$ with 
\[
\HH^2\left (\D  \backslash \bigcup_{n=1}^\infty A_n\right)=0
\] 
such that
\[
d(f(z),f(\zeta)) \leq n|z-\zeta|, \ \forall z, \zeta \in A_n.
\]}
{\sc Proof of Claim.}
By Reshetnyak's characterization of finite energy maps  (cf.~\cite{reshetnyak1}) and the equivalence of the class of Reshetnyak finite energy maps and the class of Korevaar-Schoen finite energy maps (cf.~\cite{reshetnyak2}), there exists $\phi \in L^2(\D)$ such that 
\[
|\nabla f_{z_0}|(z) \leq \phi(z) \ \mbox{ a.e.~$z \in \D$ where } f_{z_0}(\cdot) = d(f(\cdot), f(z_0)).
\]
Extend $\phi$ to $\C$ by setting it equal to zero outside of $\D$, and let $M\phi^2$ be the Hardy-Littlewood maximal function of the integrable function $\phi^2$; i.e.
\[
M\phi^2(z)=\sup_{\overline D \ni z} \frac{1}{|D|} \int_D \phi^2 dxdy
\]
where the supremum is taken over all disks $D$ such that $z \in \overline D$  and $|D|$ is the (Euclidean) volume of the disk.  For $z_0, z_1 \in \D$, let $r=|z_0-z_1|$ and $z_t=(1-t)z_0+tz_1$.
Integrating $z \in \D_r(z_{\frac{1}{2}})$ and dividing by $\frac{\pi}{r^2}$, we have 
\begin{eqnarray*}
\frac{1}{\pi r^2} \int_{\D_r(z_{\frac{1}{2}})} d(f(z_0),f(z)) \, dxdy 
& = & 
\frac{1}{\pi r^2} \int_{\D_r(z_{\frac{1}{2}})} |f_{z_0}(z_0)-f_{z_0}(z)| \, dxdy\\
 & \leq & 
\frac{1}{\pi r^2} \int_{\D_r(z_{\frac{1}{2}})}
\left( |z_0-z_1| \int_0^1  |\nabla f_{z_0}| \big( (1-t)z_0+tz\big) \, dt \right) \, dxdy
\\
& \leq & |z_0-z_1| 
\int_0^1 \left(  \frac{1}{\pi r^2}  \int_{\D_{tr}(z_{\frac{t}{2}})}  |\nabla f_{z_0}| \, dxdy\right) dt\\
& \leq  & |z_0-z_1| 
\int_0^1 t^2 \left(  \frac{1}{\pi (tr)^2}  \int_{\D_{tr}(z_{\frac{t}{2}})} \phi \, dxdy\right) dt\\
& \leq  & |z_0-z_1| 
\int_0^1 t^2 \left(  \frac{1}{\pi (tr)^2}  \int_{\D_{tr}(z_{\frac{t}{2}})} \phi^2 \, dxdy\right) dt\\
&\leq &\frac{1}{3} |z_0-z_1|  M\phi^2(z_0).
\end{eqnarray*}
Similarly, we obtain
\[
\frac{1}{\pi r^2} \int_{\D_r(z_{\frac{1}{2}})} d(f(z_1),f(z)) \, dxdy\leq \frac{1}{3}|z_0-z_1| M\phi^2(z_1),
\]
and the triangle inequality implies
\begin{eqnarray*}
d(f(z_0),f(z_1)) & = &  
\frac{1}{\pi r^2} \int_{\D_r(z_{\frac{1}{2}})}d(f(z_0),f(z_1)) \, dxdy
\\& \leq & 
\frac{1}{\pi r^2} \int_{\D_r(z_{\frac{1}{2}})}d(f(z_0),f(w)) \, dxdy
+ \frac{1}{\pi r^2} \int_{\D_r(z_{\frac{1}{2}})}d(f(z_1),f(w)) \, dxdy \\
& \leq &  |z_0-z_1| \left(M\phi^2(z_0)+M\phi^2(z_1)\right).
\end{eqnarray*}
Since $M\phi^2 \in L^1$, we have
\[
\HH^2\left(\D  \backslash \bigcup_{n=1}^\infty A_n\right)=0  \ \text{where} \ 
A_n=\{ z \in \D:  n-1 \leq 2M\phi^2(z) < n\}
\]
and by the above inequality
\[
d(f(z), f(\zeta)) \leq n |z-\zeta|, \ \forall z, \zeta \in A_n.
\]
\hfill $\Box$({\sc Claim}).  \\

Let  $\{A_n\}$ be  as in the Claim.   Fix $n$,  apply Kuratowski isometric embedding of $X$ into a Banach space $l^\infty(X)$ of bounded functions on $X$ with norm $|| \cdot ||$ and then apply the Kirsbraun theorem for Banach spaces to  extend the restriction map $f|_{A_n}:A_n \rightarrow X \subset l^{\infty}(X)$ to a Lipschitz map $\hat f:\C \rightarrow l^{\infty}(X)$.  By \cite[Theorem 2]{kirchheim}, 
MD$(\hat f, z_0)$ exists for a.e.~$z_0 \in \C$ and
\[
||\hat f(z) -\hat f(\zeta)|| - \text{MD}(\hat f,z_0)(z-\zeta) = o(|z-z_0|+|\zeta-z_0|).
\]
In particular, for $z_0, z \in A_n$, 
\[
d(f(z), f(z_0)) -\text{MD}(\hat f,z_0)(z-z_0) =o(|z-z_0|).
\]
Combined with the fact that $f=\hat f$ in $A_n$, this implies that   if $z_0 \in A_n$ is a  density 1 point of $A_n$, then $\text{MD}_{ap}(f,z_0)$ exists and  
\[
 \text{MD}_{ap}(f,z_0) = \text{MD}(\hat f,z_0).
\] 
By \cite[Theorem 1.8.1 and Lemma~1.9.5]{korevaar-schoen1},
for a.e.~$z_0 \in \D$,
\[
\left|f_*(\omega)\right|^2 = \lim_{r \rightarrow 0} \frac{d^2(f(z_0), f(z_0+r\omega))}{r^2}, \ \ \text{a.e.}~\omega \in \Sp^1.
\]
Thus, for a.e.~$z_0 \in A_n$
\[
 \text{MD}_{ap}(f,z_0)(\omega) =\left|f_*(\omega)\right|(z_0), \ \ \text{a.e.}~\omega \in \Sp^1.
 \]
The assertion now follows from the definition of $\JJ_f$ and the fact that $\HH^2(\D  \backslash \bigcup_{n=1}^\infty A_n)=0$.
\end{proof}

\end{document}